\documentclass[a4paper]{article}%
\usepackage{graphicx}
\usepackage{algorithm}
\usepackage{algorithmic}
\usepackage{amsmath}
\usepackage{amsfonts}
\usepackage{amssymb}
\usepackage{latexsym}
\usepackage{makeidx}
\usepackage[noBBpl]{mathpazo}
\usepackage{indentfirst}%
\providecommand{\U}[1]{\protect\rule{.1in}{.1in}}
\newtheorem{theorem}{Theorem}

\newtheorem{corollary}[theorem]{Corollary}

\newtheorem{lemma}[theorem]{Lemma}

\newtheorem{proposition}[theorem]{Proposition}

\newenvironment{proof}[1][Proof]{\noindent\textbf{#1.} }{\ \rule{0.5em}{0.5em}}
\begin{document}

\title{Eigenvalues and eigenfunctions of the Laplacian\\via inverse iteration with shift}
\author{{\small \textbf{Rodney Josu\'{e} BIEZUNER$^{\text{\thinspace a}}$, Grey
ERCOLE$^{\text{\thinspace a}}$, Breno Loureiro GIACCHINI$^{\text{\thinspace
b}}$, Eder Marinho MARTINS$^{\text{\thinspace c}}$}}\thanks{ \textit{E-mail
addresses:} rodney@mat.ufmg.br (R. J. Biezuner), grey@mat.ufmg.br (G. Ercole),
brenolg@ufmg.br (B. L. Giacchini) eder@iceb.ufop.br (E. Martins).}\\{\footnotesize {$^{\mathrm{a}}$} \textit{Departamento de Matem\'{a}tica -
ICEx, Universidade Federal de Minas Gerais,}}\\{\footnotesize \textit{Av. Ant\^{o}nio Carlos 6627, Caixa Postal 702,
30161-970, Belo Horizonte, MG, Brazil }} \\{\footnotesize {$^{\mathrm{b}}$} \textit{Departamento de F\'{\i}sica - ICEx,
Universidade Federal de Minas Gerais,}}\\{\footnotesize \textit{Av. Ant\^{o}nio Carlos 6627, 31270-901, Belo Horizonte,
MG, Brazil }} \\{\footnotesize {$^{\mathrm{c}}$} \textit{Departamento de Matem\'{a}tica -
ICEB, Universidade Federal de Ouro Preto,}}\\{\footnotesize \textit{Campus Universit\'{a}rio Morro do Cruzeiro, 35400-000,
Ouro Preto, MG, Brazil}} }
\maketitle

\begin{abstract}
\noindent In this paper we present an iterative method, inspired by the
inverse iteration with shift technique of finite linear algebra, designed to
find the eigenvalues and eigenfunctions of the Laplacian with homogeneous
Dirichlet boundary condition for arbitrary bounded domains $\Omega\subset
R^{N}$. This method, which has a direct functional analysis approach, does not
approximate the eigenvalues of the Laplacian as those of a finite linear
operator. It is based on the uniform convergence away from nodal surfaces and
can produce a simple and fast algorithm for computing the eigenvalues with
minimal computational requirements, instead of using the ubiquitous Rayleigh
quotient of finite linear algebra. Also, an alternative expression for the
Rayleigh quotient in the associated infinite dimensional Sobolev space which
avoids the integration of gradients is introduced and shown to be more
efficient. The method can also be used in order to produce the spectral
decomposition of any given function $u\in L^{2}(\Omega)$.

\end{abstract}


\section{Introduction}

In \cite{BEM} we introduced an iterative method for computing the first
eigenpair of the $p$-Laplacian operator $\Delta_{p}u:=\operatorname{div}%
\left(  \left\vert \nabla u\right\vert ^{p-2}\nabla u\right)  $, $p>1$, with
homogeneous Dirichlet boundary condition in a bounded domain $\Omega
\subset\mathbb{R}^{N}$, $N\geqslant1$. The technique was inspired by the
inverse power method or inverse iteration of finite linear algebra.

In the present paper we concentrate in the special case $p=2$, the Laplace
operator $\Delta$, which was superficially dealt with in \cite{BEM}. Besides
clarifying some of the arguments sketched in that paper for this case and
providing some error estimates, our main purpose in this work is to show how
inverse iteration with shift in the presence of uniform convergence can be
used in order to obtain a fast and efficient method for computing the
eigenvalues and eigenfunctions of the Laplacian operator with homogeneous
Dirichlet boundary condition for any bounded domain $\Omega$. If the
eigenvalues or at least good estimates for them are \textit{a priori} known,
the method can produce the corresponding eigenfunctions with great speed and accuracy.

The technique can alternatively also be used as a fast process to obtain the
spectral decomposition of any function $u\in L^{2}\left(  \Omega\right)  \ $
(in other words, the Fourier series of $u$).

We remark that the application of the method to the special case of the
Laplacian operator is more natural since the Laplacian is a linear operator,
$L^{2}\left(  \Omega\right)  $ is a Hilbert space and the inverse operator
$-\Delta^{-1}$ is a self-adjoint and compact operator, therefore allowing the
complete characterization of its spectral structure, as well as having the
property that its eigenfunctions constitute a basis for $L^{2}\left(
\Omega\right)  $ (except for compactness, these properties are absent in the
$p$-Laplacian when $p\neq2$).

Our approach of the inverse iteration with shift is based on the following
iterative process started by a given function $u\in L^{2}(\Omega):$%
\begin{equation}
\phi_{0}:=u\text{ \ and \ }\left\{
\begin{array}
[c]{cccc}%
-\Delta_{\sigma}\phi_{n+1} & = & \phi_{n} & \text{in }\Omega,\\
\phi_{n+1} & = & 0 & \text{on\ }\partial\Omega
\end{array}
\right.  \label{poisson}%
\end{equation}
where $\sigma>0$ is a previously fixed \textit{shift} and $\Delta_{\sigma
}:=\Delta+\sigma I$ is the corresponding \textit{shifted operator}.

The sequence $\left\{  \phi_{n}\right\}  $ is then handled in order to produce
approximations for the pair $(\lambda_{u}^{\sigma},e_{u}^{\sigma})$ where
$\lambda_{u}^{\sigma}$ denotes the eigenvalue of the Laplacian appearing in
the spectral expansion of $u$ which is closest to $\sigma,$ and $e_{u}%
^{\sigma}$ denotes the eigenfunction obtained as the projection of $u$ on the
$\lambda_{u}^{\sigma}$-eingenspace.

Inverse iteration with shift is used in finite linear algebra in order to find
the eigenvalues and eigenfunctions of a finite-dimensional linear operator. As
an eigenvalue-finding procedure it is not as efficient as other methods, such
as the $QR$ algorithm. However, if very good estimates of the eigenvalues are
known in advance, its rate of convergence to both eigenvalues and
eigenfunctions can be very fast (see \cite{Trefenthen-Bau}, for instance).

This approach can be naturally extended to self-adjoint compact linear
operators in infinite-dimensional Hilbert spaces such as the Laplacian and
those arising in Sturm-Liouville problems. In spite of this, we have not been
able to find any reference in the literature to this approach being used in
the Laplacian context.

Since there is now a vast literature concerning the search for estimates for
the eigenvalues of the Laplacian, as well as the gaps between eigenvalues
(see, for instance, \cite{Kuttler-Sigillito, Hile-Protter, Yang, Cheng-Yang};
although particularly useful in our context would be lower bounds for the
difference between consecutive eigenvalues), these results can be used in
connection with the inverse iteration with shift algorithm to find
eigenfunctions of the Laplacian in arbitrary domains, as well as better
approximations for its eigenvalues. It must be emphasized, however, that as
with the finite linear method, the inverse iteration with shift method is not
capable to find all the eigenfunctions associated to a non-simple eigenvalue.
It can only find an eigenfunction of the associated eigenspace. In the generic
sense most domains have Laplacian spectra consisting only of simple
eigenvalues (see \cite{Uhlenbeck1, Uhlenbeck2}), although many domains of
practical interest have eigenvalues with multiplicity greater than one
(usually, domains which exhibit some type of symmetry, although not all of them).

Algorithm 1 below is the simplest version of the inverse iteration with shift
algorithm for computing one specific eigenvalue and corresponding
eigenfunction of the Laplacian.

\begin{algorithm}[!ht]
\caption{Inverse Iteration with Shift for Laplacian Eigenvalue and
Eigenfunction} \small{
\begin{algorithmic}[1]
\STATE $\phi_0 = u$
\STATE Set $x_0$ \qquad\qquad\qquad\qquad\!\!\!\! (point in $\Omega$ outside nodal surfaces, randomly chosen)
\STATE Set $\sigma$ \qquad\qquad\qquad\qquad\! (shift, usually eigenvalue estimate)
\STATE Set $m$ \qquad\qquad\qquad\qquad (number of iterations, depending on method used to solve Laplacian)
\FOR{$n=0,1,2,\ldots, m$}
\STATE Solve $-\Delta_{\sigma}\phi_{n+1}=\phi_{n} \text{ in }\Omega,\text{ \ }
\phi_{n+1}=0 \text{ on }\partial\Omega$
\ENDFOR
\RETURN
$\phi_{m+1}/\left\Vert \phi_{m+1}\right\Vert_{\infty}\text{\qquad\qquad\! ($L^{\infty}$-normalized eigenfunction)}$
\RETURN
$\phi_m(x_0)/\phi_{m+1}(x_0) + \sigma\text{\qquad(eigenvalue)}$
\end{algorithmic}
}
\label{inverse iteration with shift}
\end{algorithm}

%\newpage

\noindent In principle, the function $u$ at the start of Algorithm 1 should be
chosen so that it will have components in all eigenspaces of the Laplacian and
a random choice would suffice. However, in practice, due to rounding errors
any simple function can be used. In our numerical tests (see Section
\ref{tests}), we observed that even the unit constant function could be used
in order to obtain all the eigenvalues, even in a domain where it does not
have an infinite number of them in its spectral decomposition.

In Line 6 of Algorithm 1 any PDE solver can be used. This allows one to choose
the fastest solver available for a particular domain. In Line 9 the eigenvalue
is computed according to the uniform convergence theory developed in Section
\ref{unif}. Since uniform convergence occurs away from nodal surfaces, a point
$x_{0}\in\Omega$ not in a nodal surface must be chosen; because nodal surfaces
have zero $N$-dimensional measures, a random choice will suffice in the vast
majority of cases, even taking into account that nodal surfaces change as the
computed eigenfunction changes.

In finite linear algebra, the approximation to the eigenvalue is usually
computed using the Rayleigh quotient. In our context, the eigenvalue can also
be computed via the Rayleigh quotient of the approximated eigenfunction:
\begin{equation}
\mathcal{R}\left(  \phi_{n+1}\right)  =\frac{\left\langle \nabla\phi
_{n+1},\nabla\phi_{n+1}\right\rangle _{2}}{\left\langle \phi_{n+1},\phi
_{n+1}\right\rangle _{2}}=\frac{\left\Vert \nabla\phi_{n+1}\right\Vert
_{2}^{2}}{\left\Vert \phi_{n+1}\right\Vert _{2}^{2}}=\frac{\int_{\Omega
}\left\vert \nabla\phi_{n+1}\right\vert ^{2}\,dx}{\int_{\Omega}\phi_{n+1}%
^{2}\,dx}. \label{Rayleigh}%
\end{equation}
However, due to the high oscillatory nature of high frequency eigenfunctions,
in order to accurately compute the integral of the (squared) gradient of
eigenfunctions associated with these eigenvalues, a much finer grid needs to
be used, which affects the efficiency of the method. Therefore, the Rayleigh
quotient in its original form (\ref{Rayleigh}) is not recommended for the
computation of the eigenvalues of the Laplacian, unless one is prepared to
incur the higher computational costs (see also further comments in Section
\ref{final}).

Nonetheless, an integration by parts argument produces the following
alternative expression for the Rayleigh quotient that avoids computations
involving gradients:
\begin{equation}
\mathcal{R}(\phi_{n})=\sigma+\frac{\int_{\Omega}\phi_{n}\phi_{n-1}dx}%
{\int_{\Omega}\left\vert \phi_{n}\right\vert ^{2}dx}. \label{weakRQ1}%
\end{equation}
We obtain very satisfactory results by using this sequence in Line 9 of
Algorithm 1 to compute eingenvalues in our numerical tests (see Section
\ref{tests}).

Another alternative way to compute the eigenvalues is given by the quotient
\begin{equation}
\frac{\Vert\phi_{n}\Vert_{2}}{\Vert\phi_{n+1}\Vert_{2}}+\sigma=\frac
{\int_{\Omega}\phi_{n}^{2}\,dx}{\int_{\Omega}\phi_{n+1}^{2}\,dx}+\sigma,
\label{L2quotient}%
\end{equation}
when the shift $\sigma$ is chosen below the eigenvalue. This quotient also
gives accurate approximations for the eigenvalues even using relatively coarse
meshes; the computation of the integrals of the approximated eigenfunctions,
instead of their gradients, does not appear to be significantly affected by
the use of coarse grids (see Section \ref{tests}).

We believe that, differently from what happens in finite dimensions, where
much better and faster algorithms for finding the eigenvalues of linear
operators (matrices), specially self-adjoint operators, are available, the
inverse iteration with shift algorithm can be a very competitive method for
finding the eigenvalues of the Laplacian. Typical algorithms for computing
Laplacian eigenvalues involve the discretization of the Laplacian operator and
the computation of the eigenvalues of the resulting discretization matrix.
However, since only a small number of the eigenvalues of the discretization
matrix are good approximations to the Laplacian eigenvalues (the smaller
ones), huge matrices are necessary in order to obtain a sufficiently good
number of eigenvalues. And some problems, particularly those arising in the
study of quantum billiards, demand the computation of a very large number of
eigenvalues. Needless to say, besides the requirements of memory, the size of
the matrix makes it computationally costly to find its eigenvalues (see the
classical \cite{Hackbusch} book, the review \cite{Kuttler-Sigillito} and the
more recent work \cite{Heuveline} for details).

The inverse iteration with shift method, that requires only typical relatively
modest sizes for meshes in order to solve the Poisson equation with
homogeneous Dirichlet boundary condition, can be very competitive in terms of
memory allocation and processing time. This is true even when one considers
that in order to find accurate approximations for the highest order
eigenvalues and eigenfunctions sometimes one needs to refine the mesh, due to
the increase of oscillations.

Even if good estimates for the eigenvalues of a particular domain are not
known in advance, a few iterations of inverse iteration with shift should be
able to find good approximations to them, which can work as first estimates
for the shift on a second run of the algorithm. The first choices for the
shift might be concentrated around the numbers given by Weiyl's Law (see
\cite{Weyl} or \cite{Courant-Hilbert}).

The inverse iteration with shift method can also be used in order to find the
spectral decomposition of any function $u\in L^{2}\left(  \Omega\right)  $,
that is, in order to find its projections on the Laplacian eigenspaces. One
has only to be careful to eliminate spurious projections, that is, projections
which arise from rounding errors. This can be done through computing the
Fourier coefficient associated to each eigenspace. If this coefficient becomes
less than a specified very small tolerance, this projection can be safely
discarded as arising from rounding errors. The two algorithms can be combined
together in order to simultaneously find both the desired spectral
decomposition of a given function defined on a domain and the spectrum of the
Laplacian on it. The spectral decomposition algorithm is given below
(Algorithm 2). Once again, Line 11 can be replaced by (\ref{weakRQ1}) or
(\ref{L2quotient}).

\begin{algorithm}[!ht]
\caption{Spectral Decomposition} \small{
\begin{algorithmic}[1]
\STATE $\phi_0 = u$
\STATE Set $x_0$ \qquad\qquad\qquad\qquad\qquad\qquad\qquad\!\!\! (point in $\Omega$ randomly chosen)
\FOR{$k=0,1,2,\ldots$}
\STATE $\text{Set }\sigma_k \text{\qquad\qquad\qquad\qquad\qquad\qquad\ \ \ (shift)}$
\FOR{$n=0,1,2,\ldots$}
\STATE Solve $-\Delta_{\sigma_k}\phi_{n+1}^k=\phi_{n}^k \text{ in
}\Omega,\text{ \ } \phi_{n+1}^k=0 \text{ on }\partial\Omega$
\STATE Compute $\left\langle u,\phi_{n+1}^k/\left\Vert \phi_{n+1}^k\right\Vert_{2}\right\rangle_{2}\text{\qquad\ \ (Fourier coeficient)}$
\IF{$\left|\left\langle u,\phi_{n+1}^k/\left\Vert \phi_{n+1}^k\right\Vert_{2}\right\rangle_{2}\right|>\text{ tolerance}$}
\RETURN
$\phi_{n+1}^k/\left\Vert \phi_{n+1}^k\right\Vert_{2}\text{\qquad\qquad\ ($L^2$-normalized eigenfunction)}$
\RETURN
$\left\langle u,\phi_{n+1}^k/\left\Vert \phi_{n+1}^k\right\Vert_{2}\right\rangle_{2}\text{\qquad\!\!  (Fourier coefficient)}$
\RETURN
$\phi_n(x_0)/\phi_{n+1}(x_0) + \sigma_k\text{\ \ \ \ (eigenvalue)}$\\
\textbf{break}
\ENDIF
\ENDFOR
\ENDFOR
\end{algorithmic}
} \label{spectral decomposition}
\end{algorithm}

\noindent

Although for the sake of simplicity all computations here are done for the
Laplacian, the same algorithm can be used for similar elliptic operators.

This paper is organized as follows. The inverse iteration with shift sequence
is defined in Section \ref{Sec2}, where most of the notation used in this
paper is also established. In Section 3 some well-known results concerning the
Rayleigh quotient are recalled and proven for completeness. In Sections
\ref{L2} and \ref{unif} we discuss $L^{2}$ and uniform convergence of the
inverse iterated sequence, respectively. In the short Section \ref{normaliz}
we make some considerations about the normalization process on each step of
(\ref{poisson}) which is useful for computational purposes, and in Section
\ref{tests} we present the results of some numerical experiments made in
simple domains.

Finally, in Section \ref{final} we discuss if the rate of convergence of the
method could theoretically be improved through the use of the inverse
iteration with shift given by the Rayleigh quotient, as it is standard
practice in finite linear algebra.

\section{Definition of the inverse iteration with shift sequence\label{Sec2}}

Let $\Omega$ be a bounded domain in $\mathbb{R}^{N}$ and $\mathcal{H}=\left\{
e_{k}\right\}  _{k=1}^{\infty}\subset H_{0}^{1}\left(  \Omega\right)  $ be an
orthogonal (not necessarily normalized) basis for $L^{2}\left(  \Omega\right)
$ consisting of eigenfunctions of the Laplacian operator with homogeneous
Dirichlet boundary condition, that is,
\[
\left\{
\begin{array}
[c]{ll}%
-\Delta e_{k}=\lambda_{k}e_{k} & \text{ \ \ in }\Omega,\\
e_{k}=0 & \text{ \ \ on\ }\partial\Omega,
\end{array}
\right.
\]
where $\left\{  \lambda_{k}\right\}  _{k=1}^{\infty}$ is the non-decreasing
sequence of eigenvalues of the Laplacian, counting multiplicities:
\begin{equation}
0<\lambda_{1}<\lambda_{2}\leqslant\ldots\label{eigenlap}%
\end{equation}

Let $\sigma>0$ and define the \textit{shifted operator}
\begin{equation}
\Delta_{\sigma}:=\Delta+\sigma I. \label{OpS}%
\end{equation}
It follows that $e_{k}$ is also an eigenfunction of $-\Delta_{\sigma}$
corresponding to the eigenvalue $\lambda_{k}-\sigma.$

Conversely, if $\lambda$ is an eigenvalue of $-\Delta_{\sigma}$, then
$\lambda=\lambda_{k}-\sigma$ for some $k$. Thus, the spectrum of the shifted
operator equals the spectrum of the Laplacian operator shifted to the left by
$\sigma$, while the corresponding eigenspaces are the same.

Given $u\in L^{2}\left(  \Omega\right)  $, let
\[
u=\sum_{k=1}^{\infty}\alpha_{k}e_{k}%
\]
be the Fourier expansion of $u$, so that the Fourier coefficients $\alpha_{k}$
are given by
\[
\alpha_{k}=\frac{\left\langle u,e_{k}\right\rangle _{2}}{\left\Vert
e_{k}\right\Vert _{2}^{2}}=\frac{\int_{\Omega}ue_{k}\,dx}{\int_{\Omega}%
e_{k}^{2}\,dx}.
\]
Denote by $\lambda_{u}^{1}$ the least eigenvalue whose associated eigenspace
is not orthogonal to $u.$ That is,
\[
\lambda_{u}^{1}:=\lambda_{k_{1}}\text{ \ \ where }k_{1}=\min\left\{
k:\alpha_{k}\neq0\right\}  .
\]
In other words, $\lambda_{u}^{1}$ is the first eigenvalue such that $u$ has a
non-zero component in the corresponding eigenspace. Note that if $r_{1}$ is
the multiplicity of $\lambda_{k_{1}}$ then
\[
\lambda_{u}^{1}=\lambda_{k_{1}}=\lambda_{k_{1}+1}=\cdots=\lambda_{k_{1}%
+r_{1}-1}.
\]
We will denote by $e_{u}^{1}$ the orthogonal projection of $u$ on the
$\lambda_{u}^{1}$-eigenspace, that is
\[
e_{u}^{1}:=\alpha_{k_{1}}e_{k_{1}}+\alpha_{k_{1}+1}e_{k_{1}+1}+\cdots
+\alpha_{k_{1}+r_{1}-1}e_{k_{1}+r_{1}-1}=\sum_{\lambda_{k}=\lambda_{u}^{1}%
}\alpha_{k}e_{k}.
\]
Thus, the Fourier expansion of $u$ can be rewritten as
\[
u=\sum_{\lambda_{k}\geqslant\lambda_{u}^{1}}\alpha_{k}e_{k}=e_{u}^{1}%
+\sum_{\lambda_{k}>\lambda_{u}^{1}}\alpha_{k}e_{k}=e_{u}^{1}+\sum_{k\geqslant
k_{1}+r_{1}}\alpha_{k}e_{k}.
\]
Proceeding in this way, denoting by $e_{u}^{j}$ the orthogonal projection of
$u$ on the $\lambda_{u}^{j}$-eigenspace which is the $j^{\text{th}}%
$-eigenspace not orthogonal to $u$, the eigenfunction expansion of $u$ can be
written in terms of its non-zero components in the eigenspaces of the
Laplacian as
\begin{equation}
u=\sum_{j=1}^{M}e_{u}^{j} \label{6}%
\end{equation}
where either $M$ is a positive integer or, as in most cases, $M=\infty$, and
the corresponding sequence of eigenvalues $\left\{  \lambda_{u}^{j}\right\}
_{j=1}^{\infty}$ is (strictly) increasing
\[
0<\lambda_{u}^{1}<\lambda_{u}^{2}<\ldots
\]

As is well known from the theory of compact linear operators, if $\sigma$ does
not belong to the spectrum of $-\Delta$ we have that $(-\Delta_{\sigma}%
)^{-1}:L^{2}\left(  \Omega\right)  \longrightarrow L^{2}\left(  \Omega\right)
$ is a continuous, compact and invertible operator. Therefore, whenever
$\sigma$ is not an eigenvalue of the Laplacian, we can define a sequence
$\left\{  \phi_{n}\right\}  _{n\in\mathbb{N}}\subset H_{0}^{1}\left(
\Omega\right)  $ by inverse iteration setting
\begin{equation}
\phi_{0}=u\text{ \ and \ }\left\{
\begin{array}
[c]{ll}%
-\Delta_{\sigma}\phi_{n+1}=\phi_{n} & \text{ \ \ in }\Omega,\\
\phi_{n+1}=0 & \text{ \ \ on\ }\partial\Omega.
\end{array}
\right.  \label{iteration}%
\end{equation}

For each $u\in L^{2}\left(  \Omega\right)  $ and $\sigma>0$ not in the
Laplacian spectrum consider the sequence $\left\{  \phi_{n}\right\}  $ defined
by inverse iteration in (\ref{iteration}). Since $\phi_{n+1}=\left(
-\Delta_{\sigma}\right)  ^{-1}\phi_{n}$, it follows from (\ref{6}) that the
eigenfunction expansion of $\phi_{n}$ is given by
\[
\phi_{n}=\sum_{j=1}^{M}\frac{1}{\left(  \lambda_{u}^{j}-\sigma\right)  ^{n}%
}e_{u}^{j}.
\]
Let $\lambda_{u}^{\sigma}$ be the Laplacian eigenvalue appearing in the
spectral decomposition of $u$ which is closest to $\sigma$, i.e.,
\begin{equation}
\left\vert \lambda_{u}^{\sigma}-\sigma\right\vert =\min_{j\in\mathbb{N}%
}\left\vert \lambda_{u}^{j}-\sigma\right\vert . \label{usigma}%
\end{equation}

Now let $e_{u}^{\sigma}$ denote the projection of $u$ on the $\lambda
_{u}^{\sigma}$-eigenspace. Then%
\begin{equation}
\left\langle u,e_{u}^{\sigma}\right\rangle _{2}=\left\langle e_{u}^{\sigma
},e_{u}^{\sigma}\right\rangle _{2}=\left\Vert e_{u}^{\sigma}\right\Vert
_{2}^{2} \label{eusigma}%
\end{equation}
and we can write%

\begin{align*}
\phi_{n}  &  =\frac{1}{\left(  \lambda_{u}^{\sigma}-\sigma\right)  ^{n}}%
e_{u}^{\sigma}+\sum_{\left\vert \lambda_{u}^{j}-\sigma\right\vert >\left\vert
\lambda_{u}^{\sigma}-\sigma\right\vert }^{M}\frac{1}{(\lambda_{u}^{j}%
-\sigma)^{n}}e_{u}^{j}\\
&  =\frac{1}{\left(  \lambda_{u}^{\sigma}-\sigma\right)  ^{n}}\left(
e_{u}^{\sigma}+\sum_{\left\vert \lambda_{u}^{j}-\sigma\right\vert >\left\vert
\lambda_{u}^{\sigma}-\sigma\right\vert }^{M}\left(  \frac{\lambda_{u}^{\sigma
}-\sigma}{\lambda_{j}-\sigma}\right)  ^{n}e_{u}^{j}\right)  ,
\end{align*}
or
\begin{equation}
\phi_{n}=\frac{1}{\left(  \lambda_{u}^{\sigma}-\sigma\right)  ^{n}}\left(
e_{u}^{\sigma}+\psi_{n}\right)  , \label{0}%
\end{equation}
where
\begin{equation}
\psi_{n}=\sum_{\left\vert \lambda_{u}^{j}-\sigma\right\vert >\left\vert
\lambda_{u}^{\sigma}-\sigma\right\vert }^{M}\left(  \frac{\lambda_{u}^{\sigma
}-\sigma}{\lambda_{u}^{j}-\sigma}\right)  ^{n}e_{u}^{j}. \label{3}%
\end{equation}

Throughout this paper, we will denote by $\lambda_{u}^{\tau}$ the Laplacian
eigenvalue appearing in the spectral decomposition of $u$ which is second
closest to $\sigma$, i.e.,
\begin{equation}
\left\vert \lambda_{u}^{\tau}-\sigma\right\vert =\min_{\substack{j\in
\mathbb{N}\\\lambda_{u}^{j}\neq\lambda_{u}^{\sigma}}}\left\vert \lambda
_{u}^{j}-\sigma\right\vert . \label{utau}%
\end{equation}

We will also denote by $u_{n}$ the component of $u$ in the direction of
$\dfrac{\phi_{n}}{\left\Vert \phi_{n}\right\Vert _{2}},$ that is:%
\begin{equation}
u_{n}:=\left\langle u,\frac{\phi_{n}}{\left\Vert \phi_{n}\right\Vert _{2}%
}\right\rangle _{2}\dfrac{\phi_{n}}{\left\Vert \phi_{n}\right\Vert _{2}%
}=\left(  \int_{\Omega}u\frac{\phi_{n}}{\left\Vert \phi_{n}\right\Vert _{2}%
}dx\right)  \dfrac{\phi_{n}}{\left\Vert \phi_{n}\right\Vert _{2}}. \label{un}%
\end{equation}

\section{$L^{2}$-convergence of inverse iteration with shift\label{L2}}

In this section we expose the $L^{2}$-approach of the inverse iteration with
shift. Firstly let us state some well-known results concerning the Rayleigh
quotient $\mathcal{R}:H_{0}^{1}\left(  \Omega\right)  \backslash\left\{
0\right\}  \longrightarrow\mathbb{R}$ defined by%
\begin{equation}
\mathcal{R}\left(  v\right)  =\frac{\int_{\Omega}\left\vert \nabla
v\right\vert ^{2}\,dx}{\int_{\Omega}v^{2}\,dx}. \label{originalQr}%
\end{equation}
For ease of consultation, we give short proofs of these results.

\bigskip

\begin{proposition}
A function $u\in H_{0}^{1}\left(  \Omega\right)  \backslash\left\{  0\right\}
$ is a critical point of $\mathcal{R}$ if and only if $u$ is an eigenfunction
of the Laplacian with homogeneous Dirichlet boundary condition and
$\mathcal{R}\left(  u\right)  $ is the corresponding eigenvalue.
\end{proposition}

\begin{proof}
Given $v\in H_{0}^{1}\left(  \Omega\right)  $, we have%
\[
\mathcal{R}^{\prime}\left(  u\right)  v=\frac{2}{\left\Vert u\right\Vert
_{2}^{2}}\left[  \left\langle \nabla u,\nabla v\right\rangle _{2}%
-\mathcal{R}\left(  u\right)  \left\langle u,v\right\rangle _{2}\right]  .
\]
Therefore, $\mathcal{R}^{\prime}\left(  u\right)  =0$ if and only if%
\[
\int_{\Omega}\nabla u\cdot\nabla v=\mathcal{R}\left(  u\right)  \int_{\Omega
}uv
\]
for all $v\in W_{0}^{1,2}\left(  \Omega\right)  $, that is, $u$ is a weak
solution of%
\[
\left\{
\begin{array}
[c]{ll}%
-\Delta u=\mathcal{R}\left(  u\right)  u & \text{ \ \ in }\Omega,\\
u=0 & \text{ \ \ on\ }\partial\Omega.
\end{array}
\right.
\]

\end{proof}

\bigskip

\begin{corollary}
\label{O2}The Rayleigh quotient gives a quadratically accurate estimate for
the Dirichlet Laplacian eigenvalues, that is, if $u$ is an eigenfunction of
the Laplacian with homogeneous Dirichlet boundary condition with
$\mathcal{R}\left(  u\right)  $ as the corresponding eigenvalue, then
\[
\mathcal{R}\left(  v\right)  -\mathcal{R}\left(  u\right)  =O\left(
\left\Vert v-u\right\Vert ^{2}\right)
\]
as $v\rightarrow u$ in $L^{2}(\Omega)$.
\end{corollary}

\begin{proof}
If follows immediately from Taylor's formula, since $\mathcal{R}^{\prime
}\left(  u\right)  =0.$
\end{proof}

\bigskip

Now we give the main result of this section. Before this, we note from
(\ref{0}) of Section \ref{Sec2} that
\begin{equation}
\dfrac{\phi_{n}}{\left\Vert \phi_{n}\right\Vert _{2}}=\pm\dfrac{e_{u}^{\sigma
}+\psi_{n}}{\left\Vert e_{u}^{\sigma}+\psi_{n}\right\Vert _{2}}, \label{24}%
\end{equation}
where the sign of the right-hand side will depend on whether the shift
$\sigma$ is taken above or below the eigenvalue. With respect to $n:$ if the
shift is taken below the eigenvalue, the sign will always be positive, whereas
if the shift is chosen above the eigenvalue the sign will be $(-1)^{n}.$

We note from (\ref{un}) and (\ref{24}) that
\begin{equation}
u_{n}=\left\langle u,\frac{e_{u}^{\sigma}+\psi_{n}}{\left\Vert e_{u}^{\sigma
}+\psi_{n}\right\Vert _{2}}\right\rangle _{2}\frac{e_{u}^{\sigma}+\psi_{n}%
}{\left\Vert e_{u}^{\sigma}+\psi_{n}\right\Vert _{2}} \label{un2}%
\end{equation}
and also remark that
\begin{equation}
\mathcal{R}(\phi_{n})=\sigma+\frac{\int_{\Omega}\phi_{n}\phi_{n-1}dx}%
{\int_{\Omega}\left\vert \phi_{n}\right\vert ^{2}dx}. \label{weakRQ}%
\end{equation}
This fact follows from an integration by parts after multiplying the equation
$-\Delta_{\sigma}\phi_{n}=\phi_{n-1}$ by $\phi_{n}.$ From the computational
viewpoint, (\ref{weakRQ}) has the advantage of avoiding the computation of the
gradient $\nabla\phi_{n}$ as required in (\ref{originalQr}).

\bigskip

\begin{theorem}
\label{2main}\textit{Let }$u\in L^{2}\left(  \Omega\right)  .$

\begin{enumerate}
\item[(i)] \textit{We have}
\[
\left\Vert \psi_{n}\right\Vert _{2}\leqslant\left\vert \frac{\lambda
_{u}^{\sigma}-\sigma}{\lambda_{u}^{\tau}-\sigma}\right\vert ^{n}\left\Vert
u\right\Vert _{2}.
\]
\textit{In particular, }$\psi_{n}\rightarrow0$\textit{ in }$L^{2}\left(
\Omega\right)  $\textit{ with an exponential rate.}

\item[(ii)] \textit{There exists }$n_{0}\in\mathbb{N}$\textit{ such that}
\[
\left\Vert \frac{e_{u}^{\sigma}+\psi_{n}}{\left\Vert e_{u}^{\sigma}+\psi
_{n}\right\Vert _{2}}-\frac{e_{u}^{\sigma}}{\left\Vert e_{u}^{\sigma
}\right\Vert _{2}}\right\Vert _{2}\leqslant\frac{4}{\left\Vert e_{u}^{\sigma
}\right\Vert _{2}}\left\Vert \psi_{n}\right\Vert _{2},
\]
\textit{for all }$n\geqslant n_{0}.$\textit{ In particular,}
\begin{equation}
\frac{e_{u}^{\sigma}+\psi_{n}}{\left\Vert e_{u}^{\sigma}+\psi_{n}\right\Vert
_{2}}\rightarrow\dfrac{e_{u}^{\sigma}}{\left\Vert e_{u}^{\sigma}\right\Vert
_{2}}\text{ \ \ \textit{in} }L^{2}\left(  \Omega\right)  \label{2iib}%
\end{equation}
\textit{with an exponential rate.}

\item[(iii)] \textit{The following convergences hold:}
\begin{equation}
\dfrac{\left\Vert \phi_{n}\right\Vert _{2}}{\left\Vert \phi_{n+1}\right\Vert
_{2}}\rightarrow\left\vert \lambda_{u}^{\sigma}-\sigma\right\vert ,
\label{2tolamb}%
\end{equation}%
\begin{equation}
u_{n}\rightarrow e_{u}^{\sigma}\text{ \ \ \textit{in} }L^{2}\left(
\Omega\right)  \label{2toe}%
\end{equation}
where $u_{n}$ is defined by (\ref{un}), and
\begin{equation}
\mathcal{R}\left(  \phi_{n}\right)  \rightarrow\lambda_{u}^{\sigma}
\label{2iva}%
\end{equation}
\textit{with}
\begin{equation}
\mathcal{R}\left(  \phi_{n}\right)  -\lambda_{u}^{\sigma}=O\left(  \left\vert
\frac{\lambda_{u}^{\sigma}-\sigma}{\lambda_{u}^{\tau}-\sigma}\right\vert
^{2n}\right)  . \label{rate}%
\end{equation}

\end{enumerate}
\end{theorem}

\begin{proof}
We have from (\ref{3}) that
\begin{align*}
\left\Vert \psi_{n}\right\Vert _{2}^{2}  &  =\sum_{\left\vert \lambda_{u}%
^{j}-\sigma\right\vert >\left\vert \lambda_{u}^{\sigma}-\sigma\right\vert
}^{M}\left\vert \frac{\lambda_{u}^{\sigma}-\sigma}{\lambda_{u}^{j}-\sigma
}\right\vert ^{2n}\left\Vert e_{u}^{j}\right\Vert _{2}^{2}\\
&  \leqslant\left\vert \frac{\lambda_{u}^{\sigma}-\sigma}{\lambda_{u}^{\tau
}-\sigma}\right\vert ^{2n}\sum_{\left\vert \lambda_{u}^{j}-\sigma\right\vert
>\left\vert \lambda_{u}^{\sigma}-\sigma\right\vert }^{M}\left\Vert e_{u}%
^{j}\right\Vert _{2}^{2}\leqslant\left\vert \frac{\lambda_{u}^{\sigma}-\sigma
}{\lambda_{u}^{\tau}-\sigma}\right\vert ^{2n}\left\Vert u\right\Vert _{2}^{2}.
\end{align*}
Since
\[
\left\vert \frac{\lambda_{u}^{\sigma}-\sigma}{\lambda_{u}^{\tau}-\sigma
}\right\vert <1,
\]
it follows that $\left\Vert \psi_{n}\right\Vert _{2}\rightarrow0$\ as
$n\rightarrow\infty$, which proves \textit{(i)}.

Let $n_{0}\in\mathbb{N}$ be such that
\[
\left\Vert \psi_{n}\right\Vert _{2}\leqslant\frac{1}{2}\left\Vert
e_{u}^{\sigma}\right\Vert \text{ \ \ for all }n\geqslant n_{0}.
\]
Thus, if $n\geqslant n_{0}$ it follows that
\begin{align*}
\frac{1}{2}\left\Vert e_{u}^{\sigma}\right\Vert _{2}  &  =\left\Vert
e_{u}^{\sigma}\right\Vert _{2}-\frac{1}{2}\left\Vert e_{u}^{\sigma}\right\Vert
_{2}\\
&  \leqslant\left\Vert e_{u}^{\sigma}+\psi_{n}\right\Vert _{2}+\left\Vert
\psi_{n}\right\Vert _{2}-\left\Vert \psi_{n}\right\Vert _{2}=\left\Vert
e_{u}^{\sigma}+\psi_{n}\right\Vert _{2}%
\end{align*}
and
\begin{align*}
\left\Vert \frac{e_{u}^{\sigma}+\psi_{n}}{\left\Vert e_{u}^{\sigma}+\psi
_{n}\right\Vert _{2}}-\frac{e_{u}^{\sigma}}{\left\Vert e_{u}^{\sigma
}\right\Vert _{2}}\right\Vert _{2}  &  =\left\Vert \frac{\left\Vert
e_{u}^{\sigma}\right\Vert _{2}\left(  e_{u}^{\sigma}+\psi_{n}\right)
-e_{u}^{\sigma}\left\Vert e_{u}^{\sigma}+\psi_{n}\right\Vert _{2}}{\left\Vert
e_{u}^{\sigma}+\psi_{n}\right\Vert _{2}\left\Vert e_{u}^{\sigma}\right\Vert
_{2}}\right\Vert _{2}\\
&  \leqslant\left\Vert \frac{e_{u}^{\sigma}\left(  \left\Vert e_{u}^{\sigma
}\right\Vert _{2}-\left\Vert e_{u}^{\sigma}+\psi_{n}\right\Vert _{2}\right)
+\left\Vert e_{u}^{\sigma}\right\Vert _{2}\psi_{n}}{\left(  1/2\right)
\left\Vert e_{u}^{\sigma}\right\Vert _{2}^{2}}\right\Vert _{2}\\
&  \leqslant2\frac{\left\Vert e_{u}^{\sigma}\right\Vert _{2}\left\Vert
\psi_{n}\right\Vert _{2}+\left\Vert e_{u}^{\sigma}\right\Vert _{2}\left\Vert
\psi_{n}\right\Vert _{2}}{\left\Vert e_{u}^{\sigma}\right\Vert _{2}^{2}}\\
&  =\frac{4}{\left\Vert e_{u}^{\sigma}\right\Vert _{2}}\left\Vert \psi
_{n}\right\Vert _{2},
\end{align*}
which proves \textit{(ii)}.

Since
\[
\dfrac{\left\Vert \phi_{n}\right\Vert _{2}}{\left\Vert \phi_{n+1}\right\Vert
_{2}}=\left\vert \lambda_{u}^{\sigma}-\sigma\right\vert \frac{\left\Vert
e_{u}^{\sigma}+\psi_{n}\right\Vert _{2}}{\left\Vert e_{u}^{\sigma}+\psi
_{n+1}\right\Vert _{2}},
\]
(\ref{2tolamb}) follows from \textit{(i)}.

The $L^{2}$-convergence (\ref{2toe}) follows from \textit{(\ref{un2}),
(\ref{2iib})} and\textit{ (\ref{eusigma})}.

In order to prove \textit{(\ref{2iva})} we firstly notice that
\[
\lim\left\langle \frac{\phi_{n-1}}{\left\Vert \phi_{n-1}\right\Vert _{2}%
},\frac{\phi_{n}}{\left\Vert \phi_{n}\right\Vert _{2}}\right\rangle
_{2}=\left\langle \dfrac{e_{u}^{\sigma}}{\left\Vert e_{u}^{\sigma}\right\Vert
_{2}},\dfrac{e_{u}^{\sigma}}{\left\Vert e_{u}^{\sigma}\right\Vert _{2}%
}\right\rangle _{2}\left\{
\begin{array}
[c]{r}%
1\\
-1
\end{array}%
\begin{array}
[c]{l}%
\text{if }\lambda_{u}^{\sigma}\geqslant\sigma\\
\text{if \ }\lambda_{u}^{\sigma}<\sigma
\end{array}
\right.
\]
In fact, if $\lambda_{u}^{\sigma}\geqslant\sigma$ then
\begin{align*}
\lim\left\langle \frac{\phi_{n-1}}{\left\Vert \phi_{n-1}\right\Vert _{2}%
},\frac{\phi_{n}}{\left\Vert \phi_{n}\right\Vert _{2}}\right\rangle _{2}  &
=\lim\left\langle \dfrac{e_{u}^{\sigma}+\psi_{n-1}}{\left\Vert e_{u}^{\sigma
}+\psi_{n-1}\right\Vert _{2}},\dfrac{e_{u}^{\sigma}+\psi_{n}}{\left\Vert
e_{u}^{\sigma}+\psi_{n}\right\Vert _{2}}\right\rangle _{2}\\
&  =\left\langle \dfrac{e_{u}^{\sigma}}{\left\Vert e_{u}^{\sigma}\right\Vert
_{2}},\dfrac{e_{u}^{\sigma}}{\left\Vert e_{u}^{\sigma}\right\Vert _{2}%
}\right\rangle _{2},
\end{align*}
while if $\lambda_{u}^{\sigma}<\sigma$ then
\begin{align*}
\lim\left\langle \frac{\phi_{n-1}}{\left\Vert \phi_{n-1}\right\Vert _{2}%
},\frac{\phi_{n}}{\left\Vert \phi_{n}\right\Vert _{2}}\right\rangle _{2}  &
=\lim\left\langle \left(  -1\right)  ^{n-1}\frac{e_{u}^{\sigma}+\psi_{n-1}%
}{\left\Vert e_{u}^{\sigma}+\psi_{n-1}\right\Vert _{2}},\left(  -1\right)
^{n}\frac{e_{u}^{\sigma}+\psi_{n}}{\left\Vert e_{u}^{\sigma}+\psi
_{n}\right\Vert _{2}}\right\rangle _{2}\\
&  =-\left\langle \dfrac{e_{u}^{\sigma}}{\left\Vert e_{u}^{\sigma}\right\Vert
_{2}},\dfrac{e_{u}^{\sigma}}{\left\Vert e_{u}^{\sigma}\right\Vert _{2}%
}\right\rangle _{2}.
\end{align*}
Thus, it follows from (\ref{weakRQ}) that
\[
\lim\mathcal{R}\left(  \phi_{n}\right)  =\left(  \sigma+\left\vert \lambda
_{u}^{\sigma}-\sigma\right\vert \right)  \left\{
\begin{array}
[c]{r}%
1\\
-1
\end{array}%
\begin{array}
[c]{l}%
\text{if }\lambda_{u}^{\sigma}\geqslant\sigma\\
\text{if \ }\lambda_{u}^{\sigma}<\sigma
\end{array}
\right.  =\lambda_{u}^{\sigma}.
\]

The convergence order (\ref{rate}) follows from (\ref{2toe}) and Corollary
\ref{O2}, since
\[
\mathcal{R}\left(  \phi_{n}\right)  -\lambda_{u}^{\sigma}=\mathcal{R}\left(
u_{n}\right)  -\mathcal{R}\left(  e_{u}^{\sigma}\right)  =O\left(  \left\Vert
u_{n}-e_{u}^{\sigma}\right\Vert _{2}^{2}\right)  =O\left(  \left\vert
\frac{\lambda_{u}^{\sigma}-\sigma}{\lambda_{u}^{\tau}-\sigma}\right\vert
^{2n}\right)  .
\]

\end{proof}

\section{Uniform convergence of inverse iteration with shift\label{unif}}

We begin this section by stating a $L^{\infty}$-estimate for an eigenfunction
of the Laplacian in terms of its $L^{2}$-norm. In the following, we denote by
$\left\vert \Omega\right\vert $ the Lebesgue measure of $\Omega.$

\bigskip

\begin{lemma}
Let $e\in H_{0}^{1}\left(  \Omega\right)  $ be an eigenfunction of $-\Delta$
corresponding to the eigenvalue $\lambda.$ Then\textit{ }
\begin{equation}
\left\Vert e\right\Vert _{\infty}\leqslant4^{N}\left\vert \Omega\right\vert
^{1/2}\lambda^{N/2}\left\Vert e\right\Vert _{2}. \label{1}%
\end{equation}

\end{lemma}

\begin{proof}
It is shown in \cite{Lindqvist}, without any smoothness assumption on
$\partial\Omega$, that if $e$ is an eigenfunction corresponding to a
variational eigenvalue $\lambda$ of the homogeneous Dirichlet problem for the
$p$-Laplacian then
\[
\left\Vert e\right\Vert _{\infty}\leqslant4^{N}\lambda^{N/p}\left\Vert
e\right\Vert _{L^{1}\left(  \Omega\right)  }.
\]
Choosing $p=2$, (\ref{1}) follows from H\"{o}lder's inequality.
\end{proof}

Estimates for eigenfunctions of the Laplacian with the exponent $N/2$ in the
eigenvalue replaced by $N/4$ can be found in \cite{Egorov, Yakubov1, Yakubov2,
Yakubov3}. See also \cite[Remark 5.21]{Burenkov} for more references.

The following result refers to the nondecreasing sequence (\ref{eigenlap}) of
eigenvalues of the Laplacian.

\bigskip

\begin{lemma}
If $k>N/2$, then
\begin{equation}
{\sum\limits_{j=1}^{\infty}}\frac{1}{\lambda_{j}^{k}}\leqslant\frac{NC^{k}%
}{2k-N}<\infty, \label{weyl}%
\end{equation}
where $C$ is a positive constant which depends only on $N$ and $\left\vert
\Omega\right\vert .$
\end{lemma}

\begin{proof}
It is well known (see \cite{Li Yau}, \cite{Lieb}) that
\[
\lambda_{j}\geqslant\frac{1}{C}j^{2/N},
\]
where
\begin{equation}
C=\frac{N+2}{N}\frac{\left(  \omega_{N}\left\vert \Omega\right\vert \right)
^{2/N}}{4\pi^{2}} \label{cn}%
\end{equation}
and $\omega_{N}$ is the volume of the $N$-dimensional unit ball. Hence, if
$j>N/2$ we have
\[
{\sum\limits_{j=1}^{\infty}}\lambda_{j}^{-k}\leqslant C^{k}{\sum
\limits_{j=1}^{\infty}}j^{-2k/N}<C^{k}\int_{1}^{\infty}s^{^{-2k/N}}%
ds=\frac{NC^{k}}{2k-N}.
\]

\end{proof}

In the next lemma we show that the convergence of a series formed by the
eigenvalues $\lambda_{u}^{j}$ which appear in the spectral decomposition of a
function $u\in L^{2}\left(  \Omega\right)  $ follows from the convergence of a
series formed by all eigenvalues of the Laplacian.

\bigskip

\begin{lemma}
Let $k$ be chosen so that the series $\sum\limits_{j=1}^{\infty}\frac
{1}{\lambda_{j}^{k}}$ is convergent$.$ Then the series
\[
\sum_{j=1}^{M}\frac{\left(  \lambda_{u}^{j}\right)  ^{N/2}}{\left\vert
\lambda_{u}^{j}-\sigma\right\vert ^{N/2+k+1}}%
\]
is also convergent.
\end{lemma}

\begin{proof}
Assume that the expansion of $u$ is not finite, i.e., $M=\infty$ (otherwise
the result is trivial). Since
\[
\sum_{j=1}^{\infty}\frac{1}{\left(  \lambda_{u}^{j}\right)  ^{k}}\leqslant
\sum_{j=1}^{\infty}\frac{1}{\lambda_{j}^{k}},
\]
it suffices to show that
\begin{equation}
\frac{\left(  \lambda_{u}^{j}\right)  ^{N/2}}{\left\vert \lambda_{u}%
^{j}-\sigma\right\vert ^{N/2+k+1}}\leqslant\frac{1}{\left(  \lambda_{u}%
^{j}\right)  ^{k}} \label{2}%
\end{equation}
for all sufficiently large $j$.
%Define $f:\left(  \sigma,+\infty\right)
%\longrightarrow
%\mathbb{R}
%$ by
%\[
%f(t)=\left(  \frac{t}{t-\sigma}\right)  ^{N/2+k}=\left(  1+\frac{\sigma
%}{t-\sigma}\right)  ^{N/2+k}.
%\]
%Since $f$ is decreasing, there exists $t_{0}$ such that
%\[
%f(t)<t-\sigma
%\]
%for all $t>t_{0}$.

As $\lambda_{u}^{j}\rightarrow\infty$, there exists $j_{0}$ such that
$\left\vert \lambda_{u}^{j}-\sigma\right\vert =\lambda_{u}^{j}-\sigma$ for all
$j\geqslant j_{0}$. Thus, if $j$ is sufficiently large, we can write
\[
\frac{\left(  \lambda_{u}^{j}\right)  ^{N/2+k}}{\left\vert \lambda_{u}%
^{j}-\sigma\right\vert ^{N/2+k}}=\frac{\left(  \lambda_{u}^{j}\right)
^{N/2+k}}{\left(  \lambda_{u}^{j}-\sigma\right)  ^{N/2+k}}<\lambda_{u}%
^{j}-\sigma=\left\vert \lambda_{u}^{j}-\sigma\right\vert ,
\]
whence (\ref{2}) follows.
\end{proof}

In order to prove the uniform convergence of the inverse iteration sequence
$\left\{  \phi_{n}\right\}  _{n\in\mathbb{N}}$, we return to (\ref{0}) and
write
\begin{equation}
\dfrac{\phi_{n}}{\left\Vert \phi_{n}\right\Vert _{\infty}}=\pm\dfrac
{e_{u}^{\sigma}+\psi_{n}}{\left\Vert e_{u}^{\sigma}+\psi_{n}\right\Vert
_{\infty}}. \label{4}%
\end{equation}
As in (\ref{24}), the sign of the right-hand side will depend on whether the
shift is taken above or below the eigenvalue and on $n.$

\bigskip

\begin{lemma}
\label{psitozero}The inequality
\begin{equation}
\left\Vert \psi_{n}\right\Vert _{\infty}\leqslant K\left\vert \frac
{\lambda_{u}^{\sigma}-\sigma}{\lambda_{u}^{\tau}-\sigma}\right\vert
^{n-\theta} \label{5}%
\end{equation}
holds for all sufficiently large $n$, for some $\theta>0$ and a positive
constant $K=K\left(  u,\Omega,\left\vert \lambda_{u}^{\sigma}-\sigma
\right\vert \right)  $. In particular, $\psi_{n}\rightarrow0$ uniformly in
$\overline{\Omega}$ with an exponential rate.
\end{lemma}

\begin{proof}
From (\ref{3}) and Lemma 1 we obtain
\[
\left\vert \psi_{n}\right\vert \leqslant\sum_{\left\vert \lambda_{u}%
^{j}-\sigma\right\vert >\left\vert \lambda_{u}^{\sigma}-\sigma\right\vert
}^{M}\left\vert \frac{\lambda_{u}^{\sigma}-\sigma}{\lambda_{u}^{j}-\sigma
}\right\vert ^{n}\left\vert e_{u}^{j}\right\vert \leqslant4^{N}\left\vert
\Omega\right\vert ^{1/2}\left\Vert u\right\Vert _{2}\sum_{\left\vert
\lambda_{u}^{j}-\sigma\right\vert >\left\vert \lambda_{u}^{\sigma}%
-\sigma\right\vert }^{M}\left\vert \frac{\lambda_{u}^{\sigma}-\sigma}%
{\lambda_{u}^{j}-\sigma}\right\vert ^{n}\left(  \lambda_{u}^{j}\right)
^{N/2}.
\]
But, taking $\theta=N/2+k+1$, we have
\begin{align*}
\sum_{\left\vert \lambda_{u}^{j}-\sigma\right\vert >\left\vert \lambda
_{u}^{\sigma}-\sigma\right\vert }^{M}\left\vert \frac{\lambda_{u}^{\sigma
}-\sigma}{\lambda_{u}^{j}-\sigma}\right\vert ^{n}\left(  \lambda_{u}%
^{j}\right)  ^{N/2}  &  =\left\vert \lambda_{u}^{\sigma}-\sigma\right\vert
^{\theta}\sum_{\left\vert \lambda_{u}^{j}-\sigma\right\vert >\left\vert
\lambda_{u}^{\sigma}-\sigma\right\vert }^{M}\left\vert \frac{\lambda
_{u}^{\sigma}-\sigma}{\lambda_{u}^{j}-\sigma}\right\vert ^{n-\theta}%
\frac{\left(  \lambda_{u}^{j}\right)  ^{N/2}}{\left\vert \lambda_{u}%
^{j}-\sigma\right\vert ^{\theta}}\\
&  \leqslant\left\vert \lambda_{u}^{\sigma}-\sigma\right\vert ^{\theta
}\left\vert \frac{\lambda_{u}^{\sigma}-\sigma}{\lambda_{u}^{\tau}-\sigma
}\right\vert ^{n-\theta}\sum_{\left\vert \lambda_{u}^{j}-\sigma\right\vert
>\left\vert \lambda_{u}^{\sigma}-\sigma\right\vert }^{M}\frac{\left(
\lambda_{u}^{j}\right)  ^{N/2}}{\left\vert \lambda_{u}^{j}-\sigma\right\vert
^{\theta}}\\
&  \leqslant\left\vert \frac{\lambda_{u}^{\sigma}-\sigma}{\lambda_{u}^{\tau
}-\sigma}\right\vert ^{n-\theta}\left\vert \lambda_{u}^{\sigma}-\sigma
\right\vert ^{\theta}\sum_{j=1}^{M}\frac{\left(  \lambda_{u}^{j}\right)
^{N/2}}{\left\vert \lambda_{u}^{j}-\sigma\right\vert ^{\theta}},
\end{align*}
and by Lemma 3 the last series converges. Thus, (\ref{5}) follows if we take%
\[
K=4^{N}\left\vert \Omega\right\vert ^{1/2}\left\Vert u\right\Vert
_{2}\left\vert \lambda_{u}^{\sigma}-\sigma\right\vert ^{\theta}\sum_{j=1}%
^{M}\frac{\left(  \lambda_{u}^{j}\right)  ^{N/2}}{\left\vert \lambda_{u}%
^{j}-\sigma\right\vert ^{\theta}}.
\]

\end{proof}

\bigskip

\begin{theorem}
\textit{Let }$u\in L^{2}\left(  \Omega\right)  .$\textit{ Then}

\begin{enumerate}
\item[(i)] \textit{There exists }$n_{0}\in\mathbb{N}$\textit{ such that}
\[
\left\Vert \frac{e_{u}^{\sigma}+\psi_{n}}{\left\Vert e_{u}^{\sigma}+\psi
_{n}\right\Vert _{\infty}}-\frac{e_{u}^{\sigma}}{\left\Vert e_{u}^{\sigma
}\right\Vert _{\infty}}\right\Vert _{\infty}\leqslant\frac{4}{\left\Vert
e_{u}^{\sigma}\right\Vert _{\infty}}\left\Vert \psi_{n}\right\Vert _{\infty}%
\]
\textit{for all }$n\geqslant n_{0}$\textit{. In particular,}
\[
\dfrac{e_{u}^{\sigma}+\psi_{n}}{\left\Vert e_{u}^{\sigma}+\psi_{n}\right\Vert
_{\infty}}\rightarrow\dfrac{e_{u}^{\sigma}}{\left\Vert e_{u}^{\sigma
}\right\Vert _{\infty}}\text{ \ \ \textit{uniformly in} }\Omega
\]
\textit{with an exponential rate.}

\item[(ii)] The following convergences \textit{hold}%
\begin{equation}
\dfrac{\left\Vert \phi_{n}\right\Vert _{\infty}}{\left\Vert \phi
_{n+1}\right\Vert _{\infty}}\rightarrow\left\vert \lambda_{u}^{\sigma}%
-\sigma\right\vert \label{inftolamb}%
\end{equation}
and%
\begin{equation}
u_{n}\rightarrow e_{u}^{\sigma}\text{ \ \ \textit{uniformly in} }\Omega.
\label{infuneu}%
\end{equation}

\item[(iii)] If $\mathcal{K}\subset\left\{  x:e_{u}^{\sigma}(x)\neq
0\mathit{\ }\right\}  $ is compact, then $\dfrac{\phi_{n}}{\phi_{n+1}%
}\rightarrow\lambda_{u}^{\sigma}-\sigma$\textit{ uniformly and} \textit{with
an exponential rate.}
\end{enumerate}
\end{theorem}

\begin{proof}
Let $n_{0}$ be such that $\left\Vert \psi_{n}\right\Vert _{\infty}%
\leqslant\frac{1}{2}\left\Vert e_{u}^{\sigma}\right\Vert _{\infty}$ for all
$n\geqslant n_{0}.$ Then, as in the proof of \textit{(ii)} of Theorem
\ref{2main}, we have for all $n\geqslant n_{0}$ that
\[
\frac{1}{2}\left\Vert e_{u}^{\sigma}\right\Vert _{\infty}\leqslant\left\Vert
e_{u}^{\sigma}+\psi_{n}\right\Vert _{\infty}%
\]
and
\[
\left\Vert \frac{e_{u}^{\sigma}+\psi_{n}}{\left\Vert e_{u}^{\sigma}+\psi
_{n}\right\Vert _{\infty}}-\frac{e_{u}^{\sigma}}{\left\Vert e_{u}^{\sigma
}\right\Vert _{\infty}}\right\Vert _{\infty}\leqslant\frac{4}{\left\Vert
e_{u}^{\sigma}\right\Vert _{\infty}}\left\Vert \psi_{n}\right\Vert _{\infty}.
\]
The remaining of \textit{(i)} follows from Lemma \ref{psitozero}.

In order to prove \textit{(\ref{infuneu})}, write
\[
u_{n}=\left(  \frac{\left\Vert \phi_{n}\right\Vert _{\infty}}{\left\Vert
\phi_{n}\right\Vert _{2}}\right)  ^{2}\frac{\phi_{n}}{\left\Vert \phi
_{n}\right\Vert _{\infty}}\int_{\Omega}u\frac{\phi_{n}}{\left\Vert \phi
_{n}\right\Vert _{\infty}}dx.
\]
Since%
\[
\lim\dfrac{\left\Vert \phi_{n}\right\Vert _{\infty}}{\left\Vert \phi
_{n}\right\Vert _{2}}=\lim\dfrac{\left\Vert e_{u}^{\sigma}+\psi_{n}\right\Vert
_{\infty}}{\left\Vert e_{u}^{\sigma}+\psi_{n}\right\Vert _{2}}=\dfrac
{\left\Vert e_{u}^{\sigma}\right\Vert _{\infty}}{\left\Vert e_{u}^{\sigma
}\right\Vert _{2}}%
\]
and, from \textit{(i)},%
\[
\frac{\phi_{n}}{\left\Vert \phi_{n}\right\Vert _{\infty}}\int_{\Omega}%
u\frac{\phi_{n}}{\left\Vert \phi_{n}\right\Vert _{\infty}}dx\rightarrow
\frac{e_{u}^{\sigma}}{\left\Vert e_{u}^{\sigma}\right\Vert _{\infty}}%
\int_{\Omega}u\frac{e_{u}^{\sigma}}{\left\Vert e_{u}^{\sigma}\right\Vert
_{\infty}}dx
\]
uniformly in $\Omega$, it follows from (\ref{eusigma}) that
\begin{align*}
u_{n}  &  \rightarrow\left(  \dfrac{\left\Vert e_{u}^{\sigma}\right\Vert
_{\infty}}{\left\Vert e_{u}^{\sigma}\right\Vert _{2}}\right)  ^{2}\frac
{e_{u}^{\sigma}}{\left\Vert e_{u}^{\sigma}\right\Vert _{\infty}}\int_{\Omega
}u\frac{e_{u}^{\sigma}}{\left\Vert e_{u}^{\sigma}\right\Vert _{\infty}}dx\\
&  =\frac{e_{u}^{\sigma}}{\left\Vert e_{u}^{\sigma}\right\Vert _{2}^{2}}%
\int_{\Omega}ue_{u}^{\sigma}\,dx=\frac{e_{u}^{\sigma}}{\left\Vert
e_{u}^{\sigma}\right\Vert _{2}^{2}}\left\langle u,e_{u}^{\sigma}\right\rangle
=e_{u}^{\sigma}.
\end{align*}

Since from (\ref{0}) we have%
\[
\dfrac{\left\Vert \phi_{n}\right\Vert _{\infty}}{\left\Vert \phi
_{n+1}\right\Vert _{\infty}}=\left\vert \lambda_{u}^{\sigma}-\sigma\right\vert
\frac{\left\Vert e_{u}^{\sigma}+\psi_{n}\right\Vert _{\infty}}{\left\Vert
e_{u}^{\sigma}+\psi_{n+1}\right\Vert _{\infty}},
\]
and thus (\ref{inftolamb}) also follows from Lemma \ref{psitozero}.

Now, let $\mathcal{K}\subset\subset\operatorname*{supp}e_{u}^{\sigma}$ be
compact so that
\[
m:=\min_{\mathcal{K}}\left\vert e_{u}^{\sigma}\right\vert >0
\]
and fix $n_{0}\in\mathbb{N}$ such that
\[
\left\Vert \psi_{n}\right\Vert _{\infty}<\frac{m}{2}\text{ \ \ for
all}\ n\geqslant n_{0}.
\]
Thus if $n\geqslant n_{0}$ we have on $\mathcal{K}$%
\[
\left\vert e_{u}^{\sigma}+\psi_{n}\right\vert \geqslant\left\vert
e_{u}^{\sigma}\right\vert -\left\vert \psi_{n}\right\vert \geqslant
m-\left\Vert \psi_{n}\right\Vert _{\infty}>\frac{m}{2}.
\]
Therefore, the quotient
\[
\frac{\phi_{n}}{\phi_{n+1}}=\left(  \lambda_{u}^{\sigma}-\sigma\right)
\frac{e_{u}^{\sigma}+\psi_{n}}{e_{u}^{\sigma}+\psi_{n+1}}%
\]
makes sense on $\mathcal{K}$ for all sufficiently large $n$ and again
\textit{(iii)} follows from Lemma \ref{psitozero} since
\begin{align*}
\left\vert \frac{\phi_{n}}{\phi_{n+1}}-\left(  \lambda_{u}^{\sigma}%
-\sigma\right)  \right\vert  &  =\left\vert \lambda_{u}^{\sigma}%
-\sigma\right\vert \left\vert \frac{e_{u}^{\sigma}+\psi_{n}}{e_{u}^{\sigma
}+\psi_{n+1}}-1\right\vert \\
&  =\left\vert \lambda_{u}^{\sigma}-\sigma\right\vert \left\vert \frac
{\psi_{n}-\psi_{n+1}}{e_{u}^{\sigma}+\psi_{n+1}}\right\vert \\
&  \leqslant\frac{2\left\vert \lambda_{u}^{\sigma}-\sigma\right\vert }%
{m}\left\vert \psi_{n}-\psi_{n+1}\right\vert .
\end{align*}

\end{proof}

\section{Normalization at each step \label{normaliz}}

In order to avoid numerical problems, such as overflow or underflow, it is
usual to normalize the right-hand-side function in each inverse iteration.
Although this procedure changes the sequence of iterates it maintains convergences.

In fact, let $v_{n}$ be defined by%
\begin{equation}
v_{0}=\frac{u}{\left\Vert u\right\Vert }\text{ \ and \ }\left\{
\begin{array}
[c]{ll}%
-\Delta_{\sigma}v_{n+1}=\dfrac{v_{n}}{\left\Vert v_{n}\right\Vert } & \text{
\ \ in }\Omega,\\
v_{n+1}=0 & \text{ \ \ on\ }\partial\Omega.
\end{array}
\right.  \label{normalization}%
\end{equation}
where $\left\Vert \cdot\right\Vert $ may denote the $L^{2}$-norm or the
$L^{\infty}$-norm. Then, since $\sigma$ is not an eigenvalue of $-\Delta$ it
is easy to verify that%
\begin{equation}
\phi_{n+1}=\left\Vert \phi_{n}\right\Vert v_{n+1}. \label{vphi}%
\end{equation}

Hence, for example, if one uses $\left\Vert \cdot\right\Vert =\left\Vert
\cdot\right\Vert _{2}$ then
\begin{equation}
\left\langle u,\frac{v_{n}}{\left\Vert v_{n}\right\Vert _{2}}\right\rangle
_{2}\dfrac{v_{n}}{\left\Vert v_{n}\right\Vert _{2}}=u_{n}\rightarrow
e_{u}^{\sigma} \label{unvn}%
\end{equation}
both uniformly and in $L^{2}.$ Note that this sequence is exactly the sequence
$\left\{  u_{n}\right\}  $ defined by (\ref{un}) and rewritten in terms of the
sequence $\left\{  v_{n}\right\}  .$

Moreover, in view of (\ref{vphi}) and (\ref{originalQr}) one also has
\begin{equation}
\mathcal{R}(v_{n})=\mathcal{R}(\phi_{n})\rightarrow\lambda_{u}^{\sigma}.
\label{RR}%
\end{equation}

If $\left\Vert \cdot\right\Vert =\left\Vert \cdot\right\Vert _{\infty}$ then
one has the following uniform convergence in each compact $\mathcal{K}%
\subset\left\{  x:e_{u}^{\sigma}(x)\neq0\right\}  :$%
\[
\dfrac{v_{n}}{v_{n+1}}\rightarrow\frac{\lambda_{u}^{\sigma}-\sigma}{\left\vert
\lambda_{u}^{\sigma}-\sigma\right\vert }=\left\{
\begin{array}
[c]{r}%
1\\
-1
\end{array}%
\begin{array}
[c]{l}%
\text{if }\lambda_{u}^{\sigma}>\sigma\\
\text{if \ }\lambda_{u}^{\sigma}<\sigma.
\end{array}
\right.
\]

\section{Numerical Tests\label{tests}}

In this section we present some numerical tests on the unit interval, unit
disk and unit square. The inverse iteration with shift was implemented
starting from the unit constant function $u\equiv1$ on these domains.
Eigenvalue approximations were computed by running our Algorithm 1 and taking
(in the line 9) the following sequences considered in this paper:
\[
\mu_{n}:=\dfrac{\phi_{n}}{\phi_{n+1}}(x_{0})+\sigma,\text{ \ \ }\gamma
_{n}:=\frac{\Vert\phi_{n}\Vert_{2}}{\Vert\phi_{n+1}\Vert_{2}}+\sigma
\]
and
\begin{equation}
\mathcal{R}(\phi_{n})=\sigma+\frac{\left\langle \phi_{n},\phi_{n-1}%
\right\rangle _{2}}{\left\Vert \phi_{n}\right\Vert _{2}^{2}}. \label{weakRQ2}%
\end{equation}
This last sequence is an alternative form of the Rayleigh quotient evaluated
at $\phi_{n}$ according to (\ref{weakRQ}). As previously remarked, the
numerical advantage of writing the Rayleigh quotient in this form is that one
need not compute the gradient $\nabla\phi_{n}.$

Taking into account (\ref{RR}) we also compute eigenvalue approximations using
the sequence $\mathcal{R}(v_{n})$ where $v_{n}$ is defined by
(\ref{normalization}). Namely, we use the following alternative expression for
this sequence:%
\begin{equation}
\mathcal{R}(v_{n})=\sigma+\frac{\left\langle v_{n},v_{n-1}\right\rangle _{2}%
}{\left\Vert v_{n}\right\Vert _{2}^{2}}. \label{weakRQv}%
\end{equation}
Notwithstanding the (theoretical) equality $\mathcal{R}(v_{n})=\mathcal{R}%
(\phi_{n})$, our tests indicate some remarkable numerical differences between them.

The non-normalization of the function $\phi_{n}$ on each step in Algorithm 1
tends to attribute to it smaller values at each iteration. Thus, as a
numerical phenomenon, the quotient $\phi_{n}\slash \phi_{n+1}$ tends to assume
the value 1, what makes the sequences $\mu_{n},\ \gamma_{n} \text{ and }
\mathcal{R}(\phi_{n})$ converge to $\sigma+ 1$.

This phenomenon may restrict the use of these sequences. However, handled with
care by controlling the number of iterations and points in the grid, they can
provide good approximations to the eigenvalues, as our numerical tests indicate.

On the other hand, the sequence $\mathcal{R}(v_{n})$ seems to be more robust,
since we did not observe this phenomenon when using it. Another indication of
its numerical stability is its tendency to better capture the correct
eigenvalues (i. e. those belonging to the spectrum of the starting function),
than the other sequences.

The graphs of the eigenfunction approximations were constructed by using the
sequence $u_{n}$ in (\ref{unvn}). In Algorithm 2, these approximations can be
viewed as the combination of lines 9 and 10 with $\phi_{n+1}^{k}$ replaced by
$v_{n+1}^{k}$ (normalizing on each step).

To show the efficiency of inverse iteration with shift we used neither the
most efficient available method for solving the underlying differential
equation nor a fine grid, but one of the most basic methods, finite
differences, and a relatively coarse grid. Integrals were computed via the
Simpson composite method.

\subsection{Eigenvalues and eigenfunctions for the unit interval $[0,1]$}

In this case, (\ref{iteration}) becomes the following boundary value problem
\[
\left\{
\begin{array}
[c]{l}%
-\phi_{n+1}^{\prime\prime}-\sigma\phi_{n+1}=\phi_{n}\\
\phi_{n+1}(0)=\phi_{n+1}(1)=0.
\end{array}
\right.
\]
One can verify that $\lambda_{k}=k^{2}\pi^{2}$ and that the function
$u\equiv1$ does not have components corresponding to $\lambda_{k}$ for $k$ even.

We present in Table 1 exact and approximated eigenvalues of the Laplacian on
the unit interval. The shift was set to the corresponding exact eigenvalue
minus $0.1$ (that is, $\sigma_{k}:=\lambda_{k}-0.1$), and a grid of only 101
nodes was used. As shown in the last column, the sequence $\mathcal{R}(v_{n})$
seems to capture only the eigenvalues $\lambda_{u}^{\sigma}$ that appear on
the spectral decomposition of the function $u\equiv1.$ In fact, we note that
for even values of $k$ the shift is closer to $\lambda_{k}$ than
$\lambda_{k-1}$ and even so the corresponding sequence $\mathcal{R}(v_{n})$
converges to $\lambda_{k-1}$ which is the correct eigenvalue.

In Figure 1 we present the graphs of the first eight approximated eigenfunctions.

\begin{center}
\begin{table}[h]
%\label{intervalTable}
%\par
\par
\begin{center}%
\begin{tabular}
[c]{clllll}\hline
$k$ & $\lambda_{k}$ & $\mu_{10}$ & $\gamma_{10}$ & $\mathcal{R}(\phi_{10})$ &
$\mathcal{R}(v_{10})$\\\hline
$1$ & $9.8696$ & $9.8688$ & $9.8688$ & $9.8688$ & $9.8688$\\
$2$ & $39.4784$ & $39.4654$ & $39.4654$ & $39.4654$ & $9.8691$\\
$3$ & $88.8264$ & $88.7607$ & $88.7607$ & $88.7607$ & $88.7607$\\
$4$ & $157.914$ & $157.706$ & $157.921$ & $157.706$ & $89.1623$\\
$5$ & $246.740$ & $246.233$ & $247.047$ & $246.233$ & $246.233$\\
$6$ & $355.306$ & $354.255$ & $356.157$ & $356.206$ & $252.178$\\
$7$ & $483.611$ & $481.665$ & $485.356$ & $481.665$ & $481.665$\\
$8$ & $631.655\hspace{0.5in}$ & $628.335\hspace{0.5in}$ & $634.773\hspace
{0.5in}$ & $632.555\hspace{0.5in}$ & $514.785$\\\hline
\end{tabular}
\end{center}
\caption{Exact and approximated eigenvalues on the unit interval $[0,1]$
obtained from the inverse iteration with shift starting from the unit
function. The shift $\sigma_{k}:=\lambda_{k}-0.1$ and a grid containing 101
nodes were used.}%
\end{table}
\end{center}

\begin{figure}[ptb]
\begin{center}
\includegraphics[height=6cm]{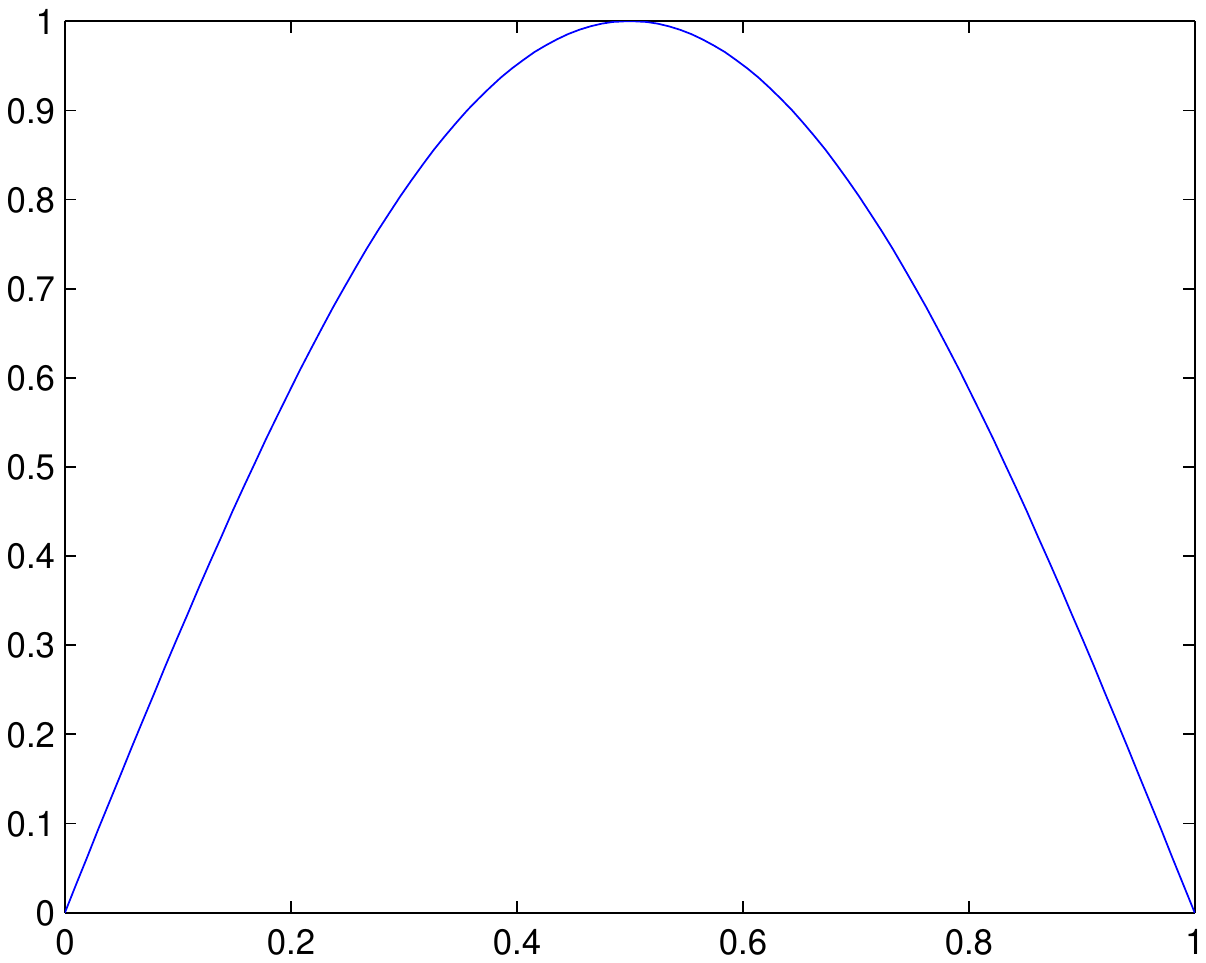}
\includegraphics[height=6cm]{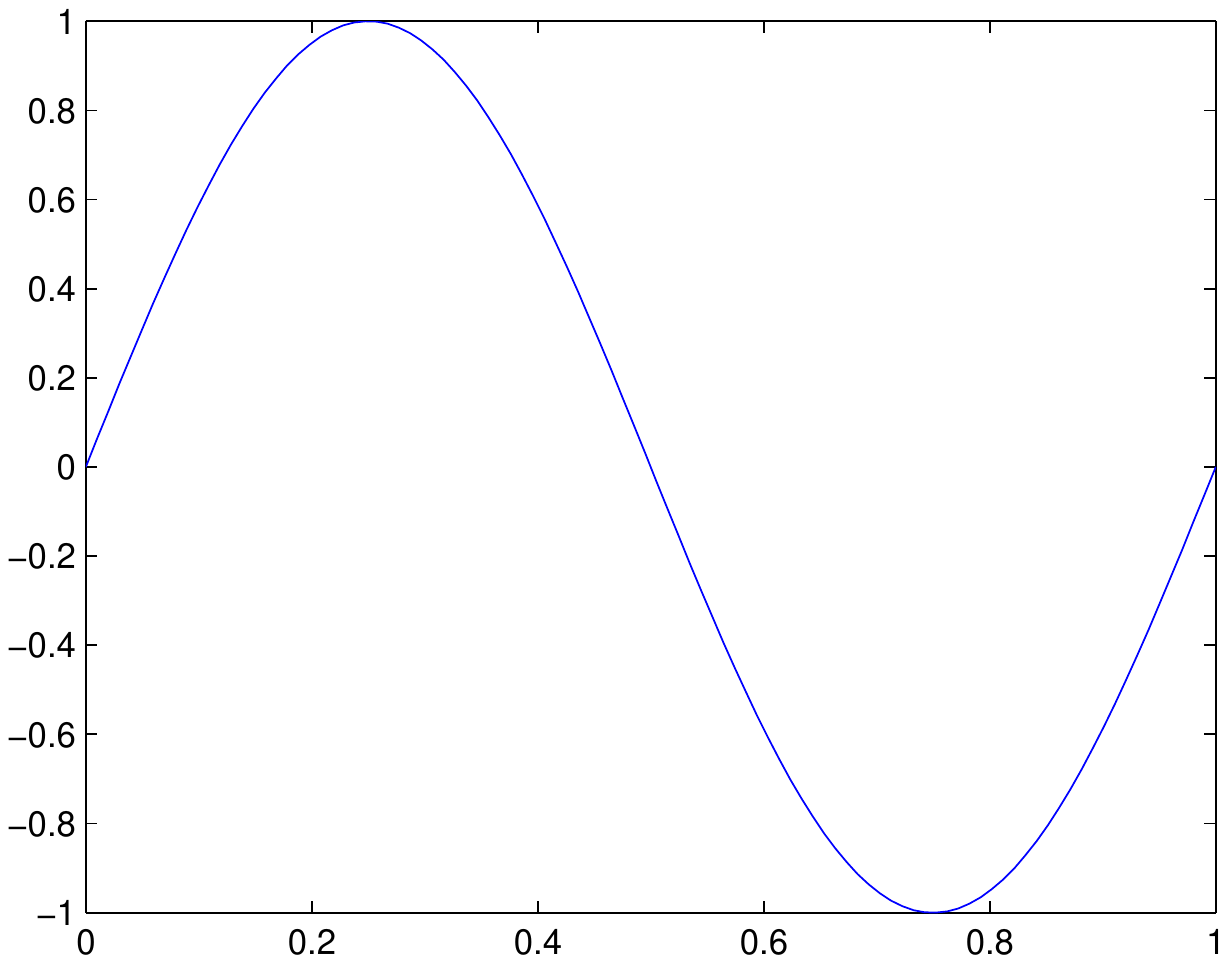}
\includegraphics[height=6cm]{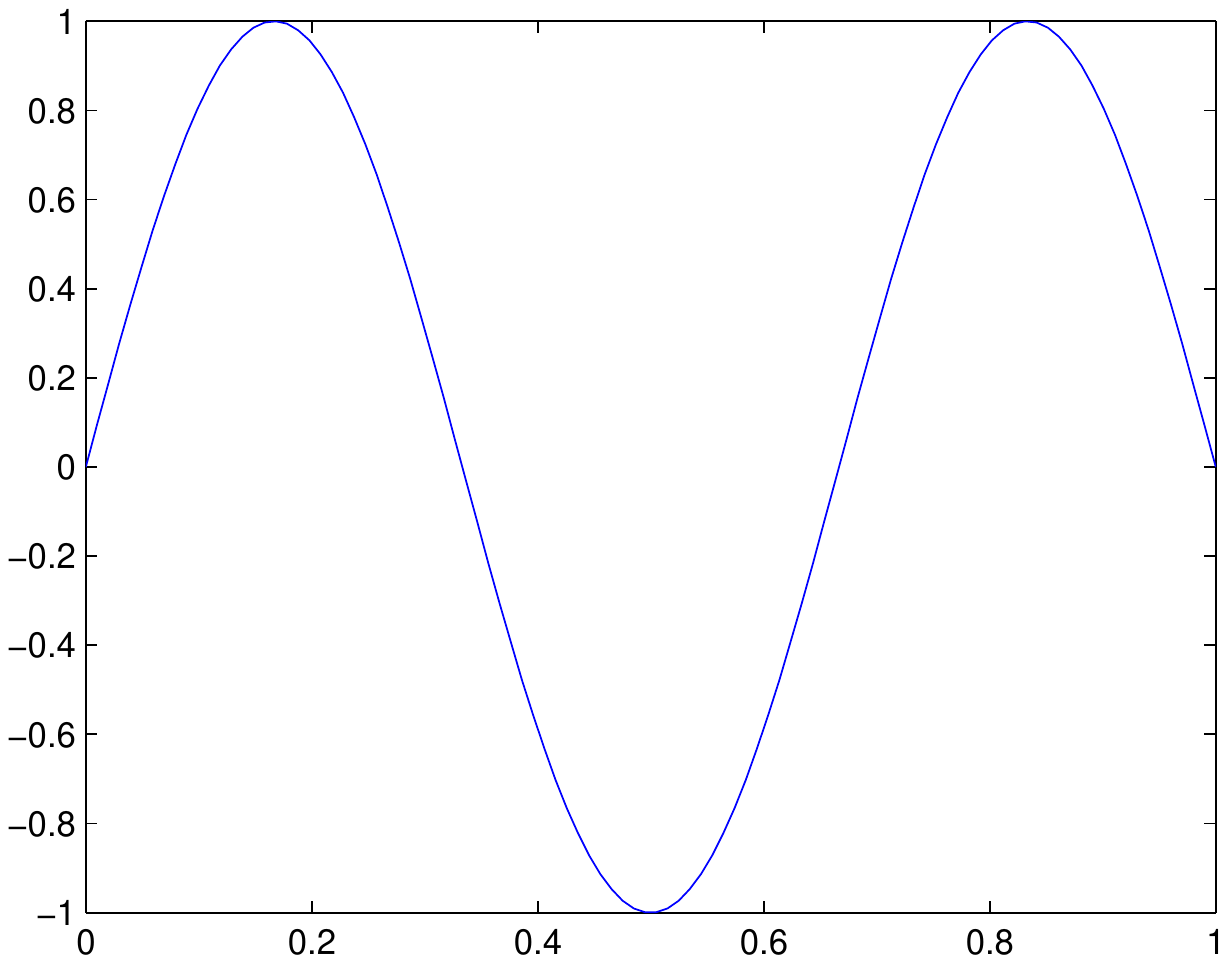}
\includegraphics[height=6cm]{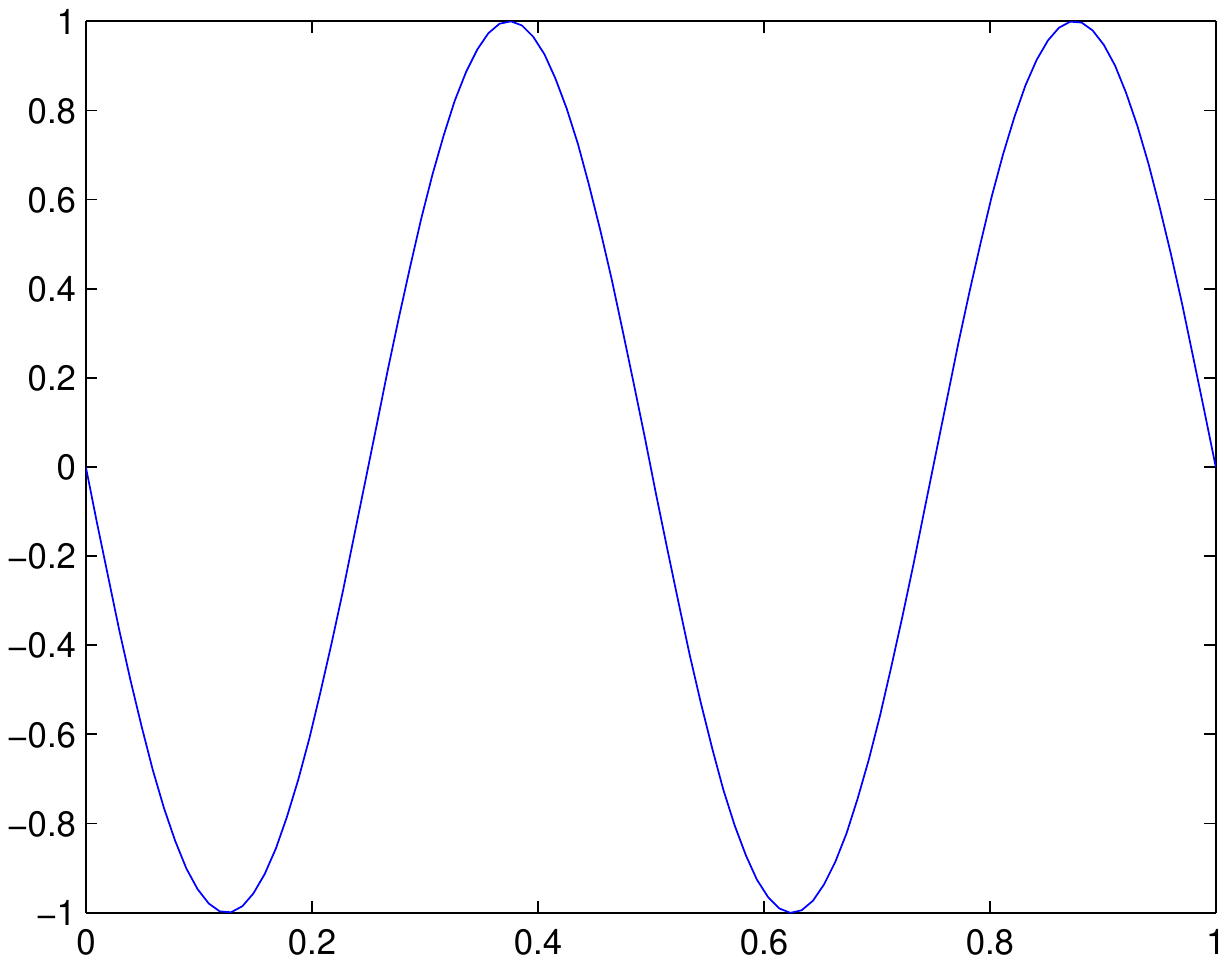}
\vspace{-2cm}
\includegraphics[height=6cm]{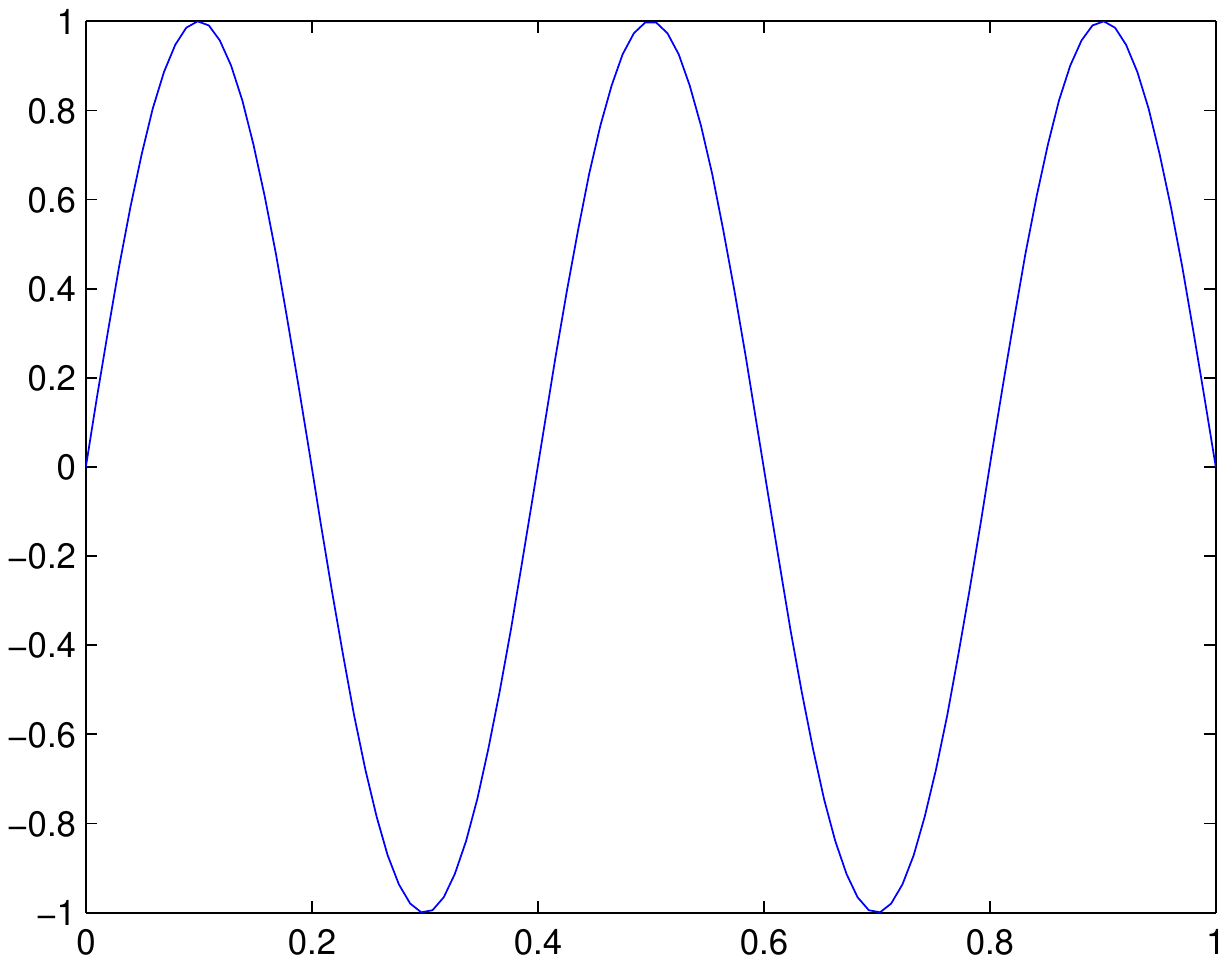}
\includegraphics[height=6cm]{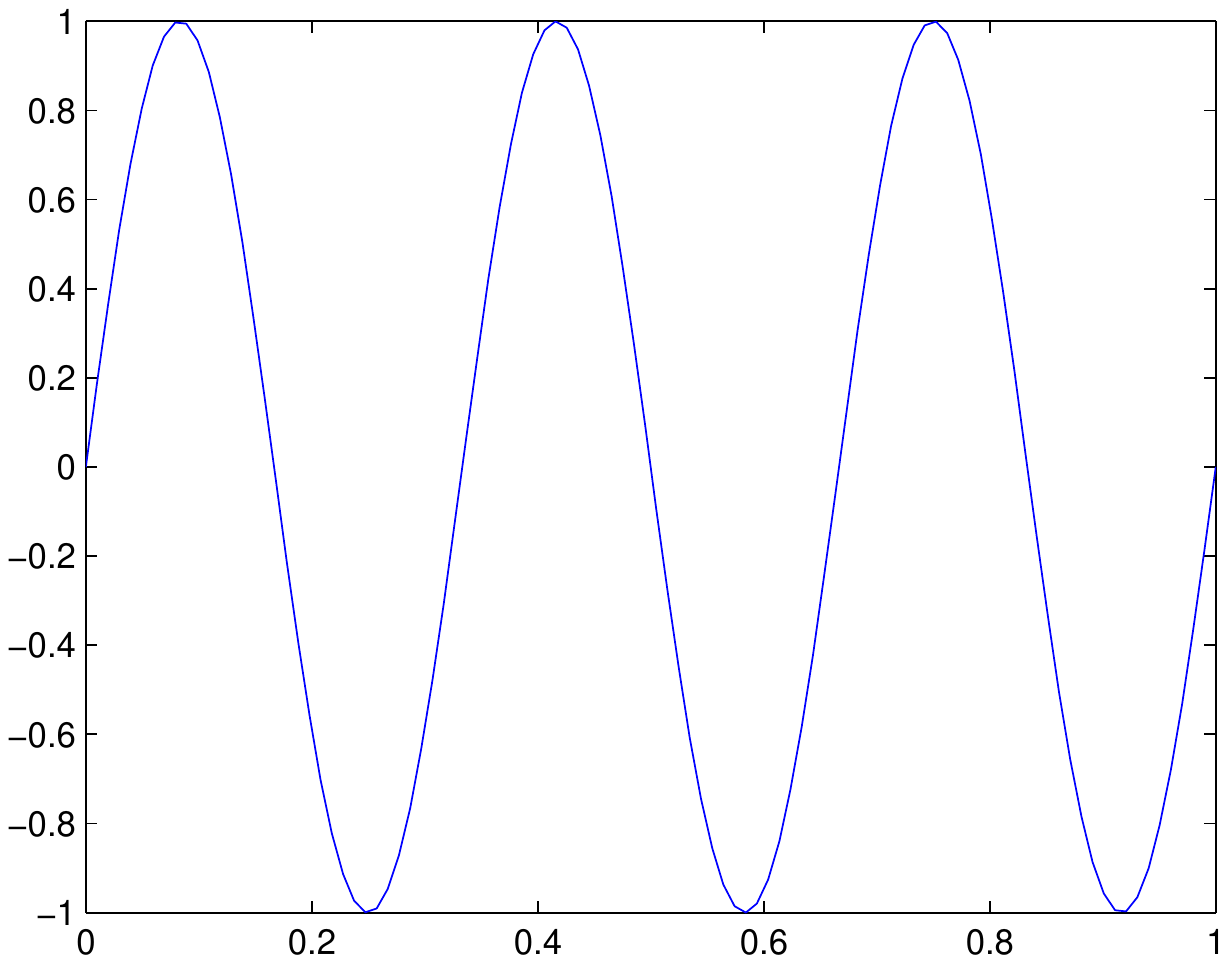}
\includegraphics[height=6cm]{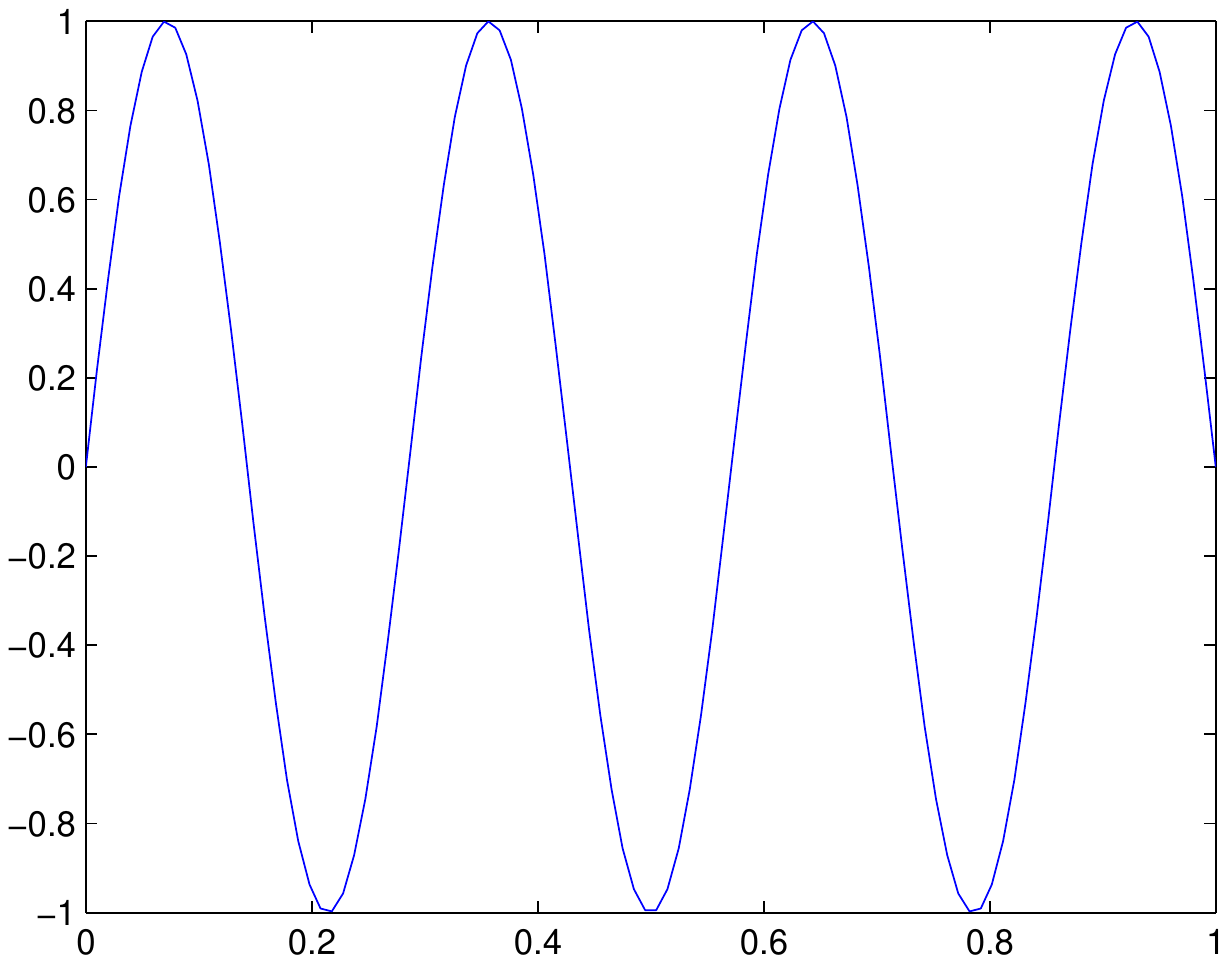}
\includegraphics[height=6cm]{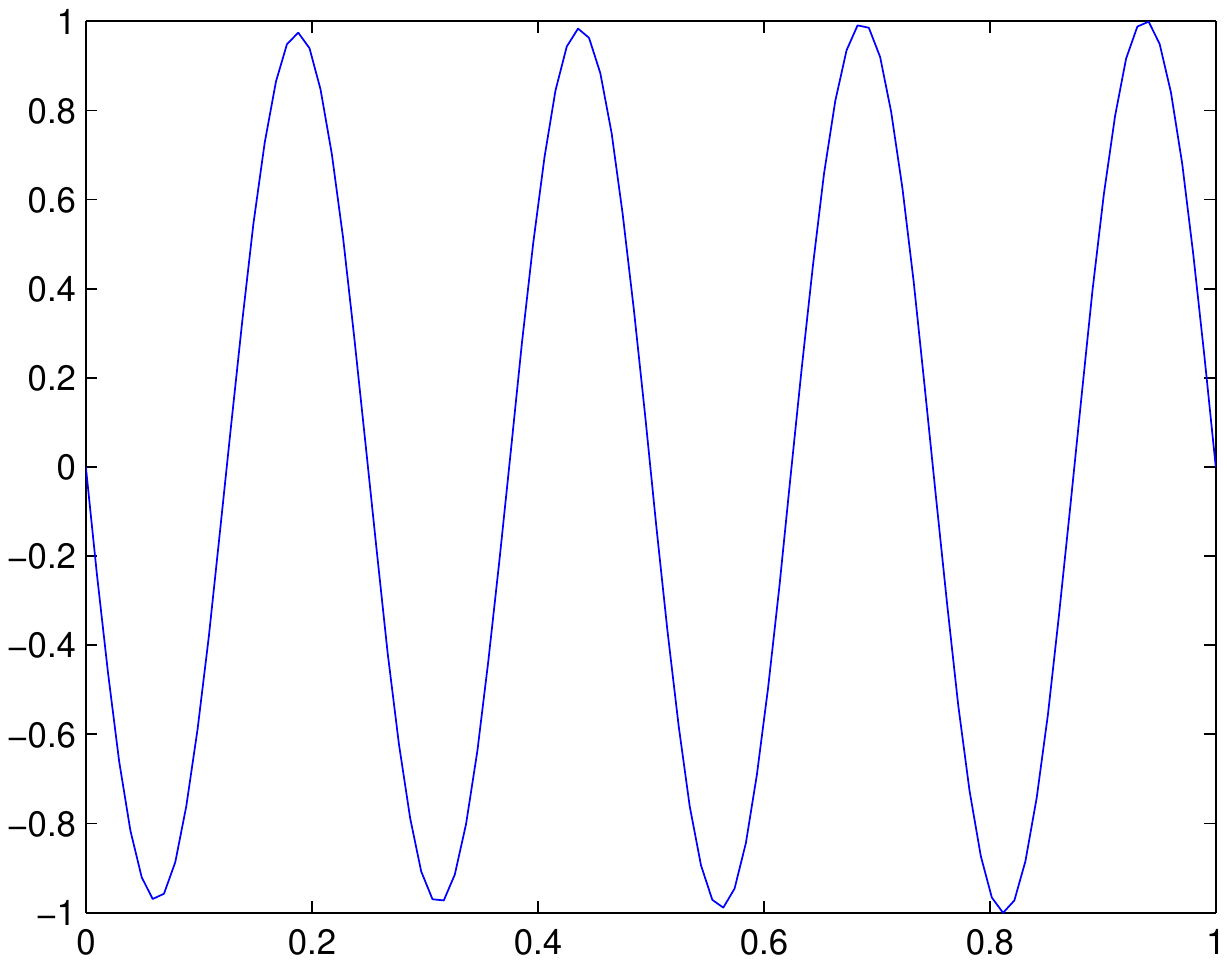}
\end{center}
\caption{First eight approximated eigenfunctions of the Laplacian on $[0,1]$
obtained from the inverse iteration with shift starting from the unit
function.}%
\end{figure}

Table 1 also exemplifies the numerical convergence to $\sigma+1$ of the
non-normalized sequence $\mathcal{R}(\phi_{n})$, which happened to the
approximations of $\lambda_{6}$ and $\lambda_{8}$. For the first one, for
example, the closest approximation achieved is $\mathcal{R}(\phi_{4}%
)=280.278$. For $n>4$ the quotient collapses to 1 and the result is spurious.
In order to compute a correct approximation of this eigenvalue using this
sequence, a finer grid should be used.

In Table 2 we show the result of calculating the first 1,500 eigenvalues of
the unit interval using the normalized sequence $\mathcal{R}(v_{n})$ with 30
iterations and a grid containing 10,001 nodes. The relative error between the
computed eigenvalue $\mathcal{R}(v_{n})$ and the exact eigenvalue $\lambda$ is
defined by%
\[
\epsilon(\mathcal{R}(v_{n}),\lambda)=\Big\vert\dfrac{\mathcal{R}%
(v_{n})-\lambda}{\lambda}\Big\vert.
\]

The shift used to make Table 2 is $\sigma_{k}:=0.99\lambda_{k}$, which has an
initial relative error of 1\%. Such error is huge for great eigenvalues and
the interval $(\sigma_{k},\lambda_{k})$ may contain many other eigenvalues.
Thus, it is likely to happen that this shift makes the sequence converge to an
eigenvalue $\lambda_{r}$ different from $\lambda_{k}$. With this in mind, we
considered that $\mathcal{R}(v_{30})$ correctly approximated an eigenvalue
$\lambda_{r}$ if $\left\vert \lambda_{r}-\mathcal{R}(v_{30})\right\vert
=\min_{s}\left\vert \lambda_{s}-\mathcal{R}(v_{30})\right\vert $, and that
$\mathcal{R}(v_{30})$ converged to $\lambda_{r}$ if their relative error is
less than $10^{-3}$.

%a tolerance value.

Table 2 reveals that among the 1,500 shifts used, all of them approximated an
eigenvalue with a relative error of order of magnitude $10^{-3}$, and that
1208 converged with a relative error less than $10^{-3}$.

\begin{center}
\begin{table}[h]
\begin{center}%
\begin{tabular}
[c]{cllllll}\hline
$\epsilon$ & $10^{-3}$ & $10^{-4}$ & $10^{-5}$ & $10^{-6}$ & $10^{-7}$ &
$\leqslant10^{-8}$\\\hline
$N_{\lambda}$ & $292\hspace{0.5in}$ & $1004\hspace{0.5in}$ & $156\hspace
{0.5in}$ & $39\hspace{0.5in}$ & $6\hspace{0.5in}$ & $3$\\\hline
\end{tabular}
\end{center}
\caption{Number $N_{\lambda}$ of approximates with relative error of order
$\epsilon$ for $\sigma_{k}:= 0.99 \lambda_{k}$, grid containing 10,001 points
and 30 iterations.}%
\end{table}
\end{center}

\noindent Among the 1208 converged approximations, 547 refer to eigenvalues
$\lambda_{k}$ with $k$ even. This shows that the sequence $\mathcal{R}(v_{n})$
can converge to an eigenvalue that does not belong to the spectrum of the unit function.

In order to better understand convergence of shifts with large initial
relative error, we used a grid of 10,001 nodes and 30 iterations of the
sequence $\mathcal{R}(v_{n})$ with shifts $\sigma_{k}:= 0.5(\lambda
_{k+1}+\lambda_{k})$, where $k$ runs from 1 to 100. The result is presented in
Table 3. As we can see, despite the shift being located exactly between the
eigenvalues, only one did not converge to an eigenvalue.

\begin{center}
\begin{table}[h]
%\label{intervalTable}
%\par
\par
\begin{center}%
\begin{tabular}
[c]{clllll}\hline
$\epsilon$ & $10^{-4}$ & $10^{-5}$ & $10^{-6}$ & $10^{-7}$ & $\leqslant
10^{-8}$\\\hline
$N_{\lambda}$ & $1\hspace{0.5in}$ & $66\hspace{0.5in}$ & $25\hspace{0.5in}$ &
$4\hspace{0.5in}$ & $4$\\\hline
\end{tabular}
\end{center}
\caption{Number $N_{\lambda}$ of approximates with relative error of order
$\epsilon$ for $\sigma_{k}:= 0.5(\lambda_{k+1}+\lambda_{k})$, grid containing
10,001 points and 30 iterations.}%
\end{table}
\end{center}

Another numerical experiment using shifts with large initial error was done
with randomly chosen shifts. We generated 100 random numbers (shifts) on the
interval $(0,\lambda_{50})$ and used a grid of 10,001 nodes and 30 iterations.
Table 4 shows the initial relative errors of the shifts, as well as the errors
after 30 iterations of sequence $\mathcal{R}(v_{n})$. As we can see, only one
of these shifts did not converge to an eigenvalue, and most of them converged
with a relative error of order $10^{-5}$.

\begin{center}
\begin{table}[h]
%\label{intervalTable}
%\par
\par
\begin{center}%
\begin{tabular}
[c]{cllllll}\hline
$\epsilon$ & $\geqslant10^{-2}$ & $10^{-3}$ & $10^{-4}$ & $10^{-5}$ &
$10^{-6}$ & $\leqslant10^{-7}$\\\hline
$N_{\sigma}$ & $38\hspace{0.5in}$ & $60\hspace{0.5in}$ & $2\hspace{0.5in}$ &
$0\hspace{0.5in}$ & $0$ & $0$\\
$N_{\lambda}$ & $0\hspace{0.5in}$ & $1\hspace{0.5in}$ & $2\hspace{0.5in}$ &
$80\hspace{0.5in}$ & $16\hspace{0.5in}$ & $1$\\\hline
\end{tabular}
\end{center}
\caption{Numbers $N_{\sigma}$ of shifts and $N_{\lambda}$ of approximates with
relative errors of order $\epsilon$. Here $\sigma$ was randomly chosen on the
interval $(0, \lambda_{50})$. A grid containing 10,001 points and 30
iterations were used.}%
\end{table}
\end{center}

\subsection{Radial eigenvalues and eigenfunctions for the unit disk}

We computed only the radial eigenfunctions for the unit disk $\Omega=\left\{
x\in\mathbb{R}^{2}:\left\vert x\right\vert \leq1\right\}  .$ In this case
$\phi_{n}=\phi_{n}(r)$ where $r=\left\vert x\right\vert ,$ and
(\ref{iteration}) becomes the Sturm-Liouville problem type%

\[
\left\{
\begin{array}
[c]{l}%
-\dfrac{(r\phi_{n+1}^{\prime})^{\prime}}{r}-\sigma\phi_{n+1}=\phi_{n},\text{
\ \ }0<r<1\\
\\
\phi_{n+1}^{\prime}(0)=0=\phi_{n+1}(1).
\end{array}
\right.
\]
Note that the function $u\equiv1$ has components in all radial eigenspaces. In
fact, if $e=e(r)$ denotes a radial eigenfunction corresponding to an
eigenvalue $\lambda>0$ in $\Omega$ then
\[
\left\{
\begin{array}
[c]{l}%
\Delta_{r}e=\dfrac{(re^{\prime})^{\prime}}{r}=-\lambda e\text{ \ \ }0<r<1\\
\\
e^{\prime}(0)=0=e(1).
\end{array}
\right.
\]
Therefore,
\begin{align*}
\left\langle u,1\right\rangle _{2}  &  =\int_{\left\vert x\right\vert \leq
1}e(\left\vert x\right\vert )dx\\
&  =\int_{0}^{1}\int_{\left\vert x\right\vert =r}e(r)dS_{x}dr=2\pi\int_{0}%
^{1}e(r)rdr=-\frac{2\pi}{\lambda}\int_{0}^{1}(re^{\prime}(r))^{\prime
}dr=-\frac{2\pi}{\lambda}e^{\prime}(1)\neq0
\end{align*}
because of the uniqueness of the initial value problems for the ODE above at
$r=1.$

We present in Table 5 the exact and approximated first eight radial
eigenvalues for the Laplacian on the unit disk, calculated using the shift
$\sigma_{k}:=\lambda_{k}-0.1$ and a grid containing 201 nodes.

\begin{table}[h]
\label{diskTable}
\par
\begin{center}%
\begin{tabular}
[c]{llllll}\hline
$k$ & $\lambda_{k}$ & $\mu_{10}$ & $\gamma_{10}$ & $\mathcal{R}(\phi_{10})$ &
$\mathcal{R}(v_{10})$\\\hline
$1$ & $5.7831$ & $5.7834$ & $5.7834$ & $5.7392$ & $5.7392$\\
$2$ & $30.4713$ & $30.4698$ & $30.4698$ & $30.4396$ & $30.4396$\\
$3$ & $74.887$ & $74.865$ & $74.865$ & $74.847$ & $74.847$\\
$4$ & $139.040$ & $138.942$ & $138.942$ & $138.942$ & $138.942$\\
$5$ & $222.932$ & $222.646$ & $223.018$ & $222.674$ & $222.674$\\
$6$ & $326.563$ & $325.901$ & $327.025$ & $325.968$ & $326.563$\\
$7$ & $449.934$ & $448.611$ & $451.057$ & $448.729$ & $448.729$\\
$8\hspace{0.3in}$ & $593.043\hspace{0.3in}$ & $590.663\hspace{0.3in}$ &
$595.223\hspace{0.3in}$ & $590.836\hspace{0.3in}$ & $590.836$\\\hline
\end{tabular}
\end{center}
\caption{The first eight Laplacian radial eigenvalues on the unit disk
obtained from the inverse iteration with shift starting from the unit
function.}%
\end{table}

The graphs of the first eight approximated radial eigenfunctions of the
Laplacian obtained by our Algorithm 2 are displayed in Figure 2.

\begin{figure}[ptb]
\begin{center}
\includegraphics[height=6cm]{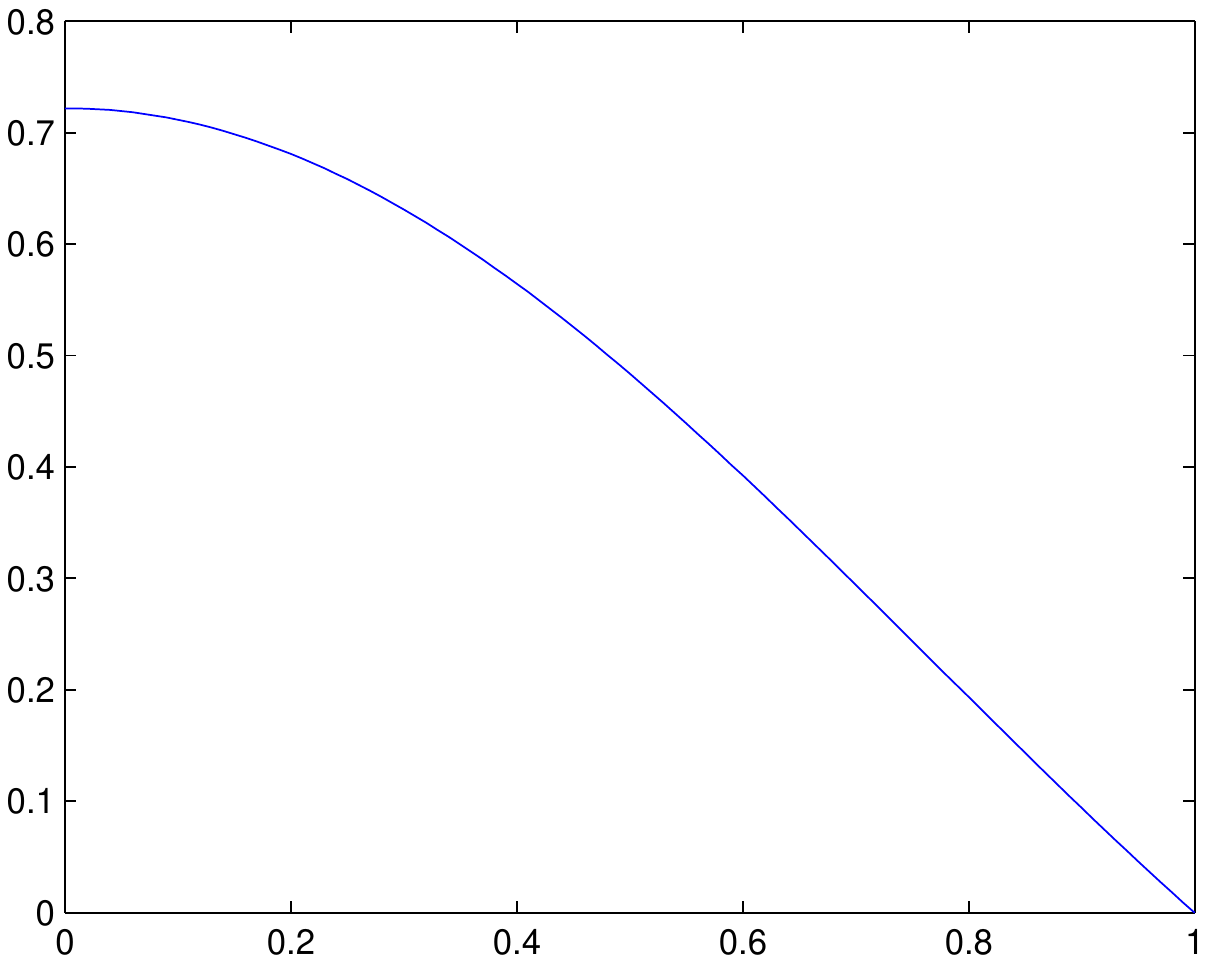}
\includegraphics[height=6cm]{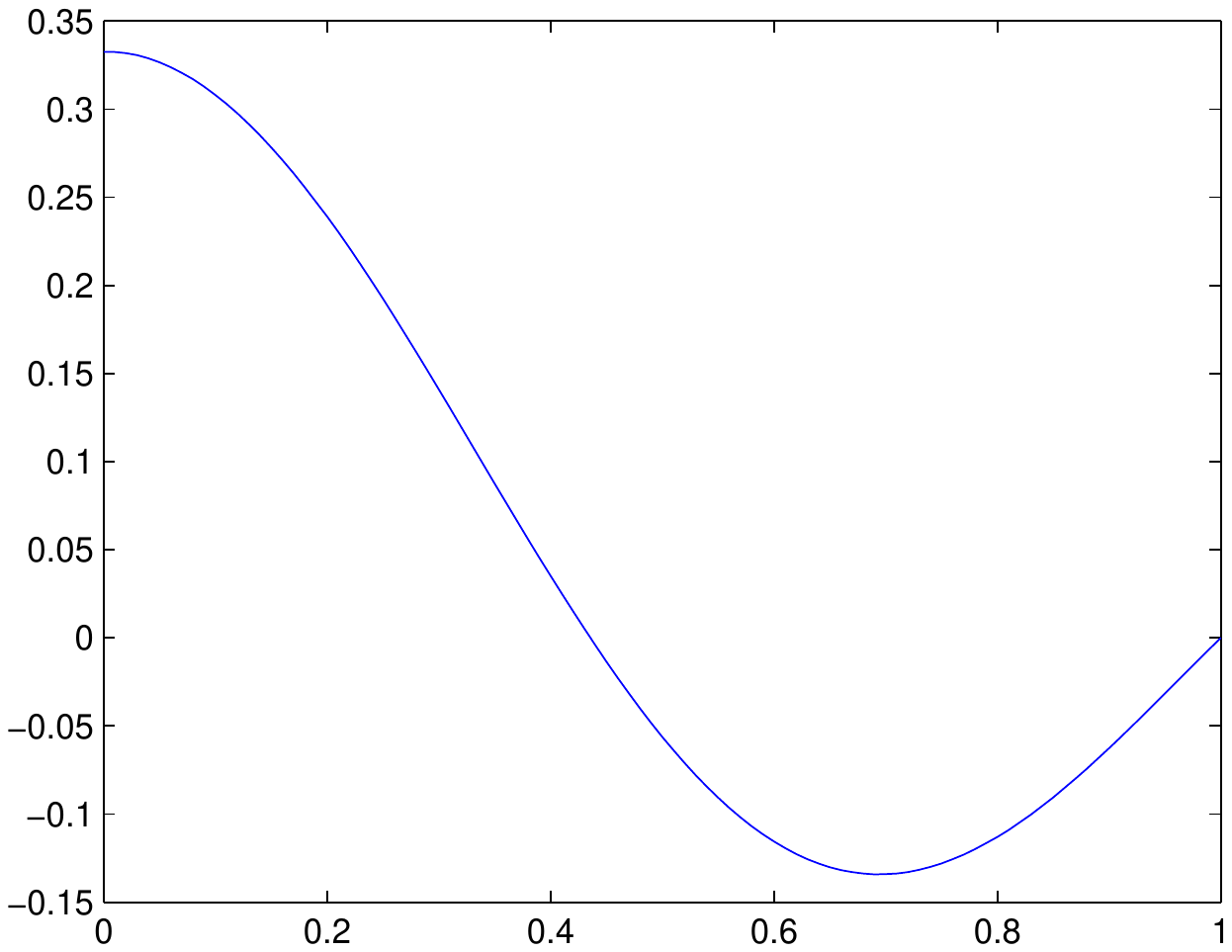}
\includegraphics[height=6cm]{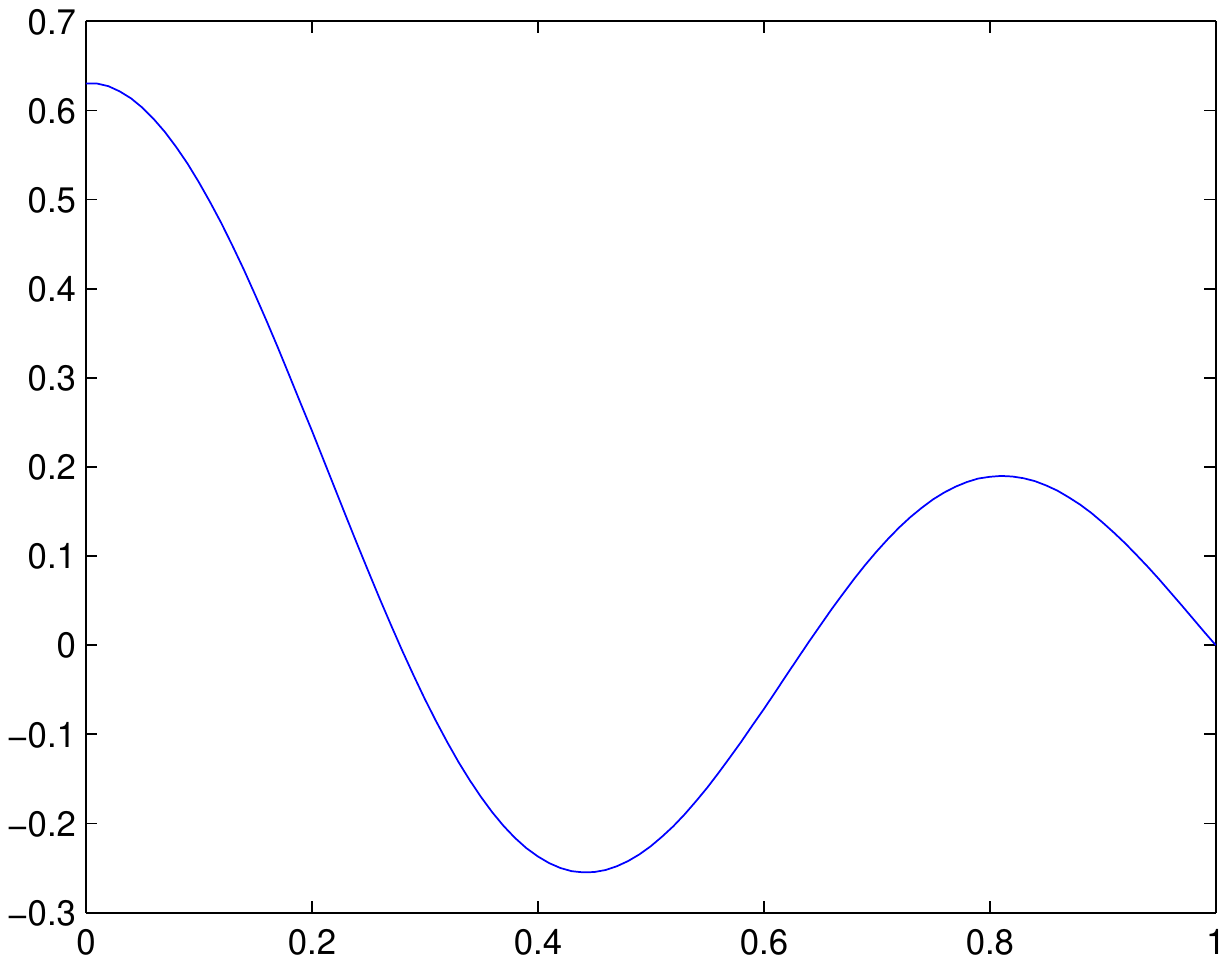}
\includegraphics[height=6cm]{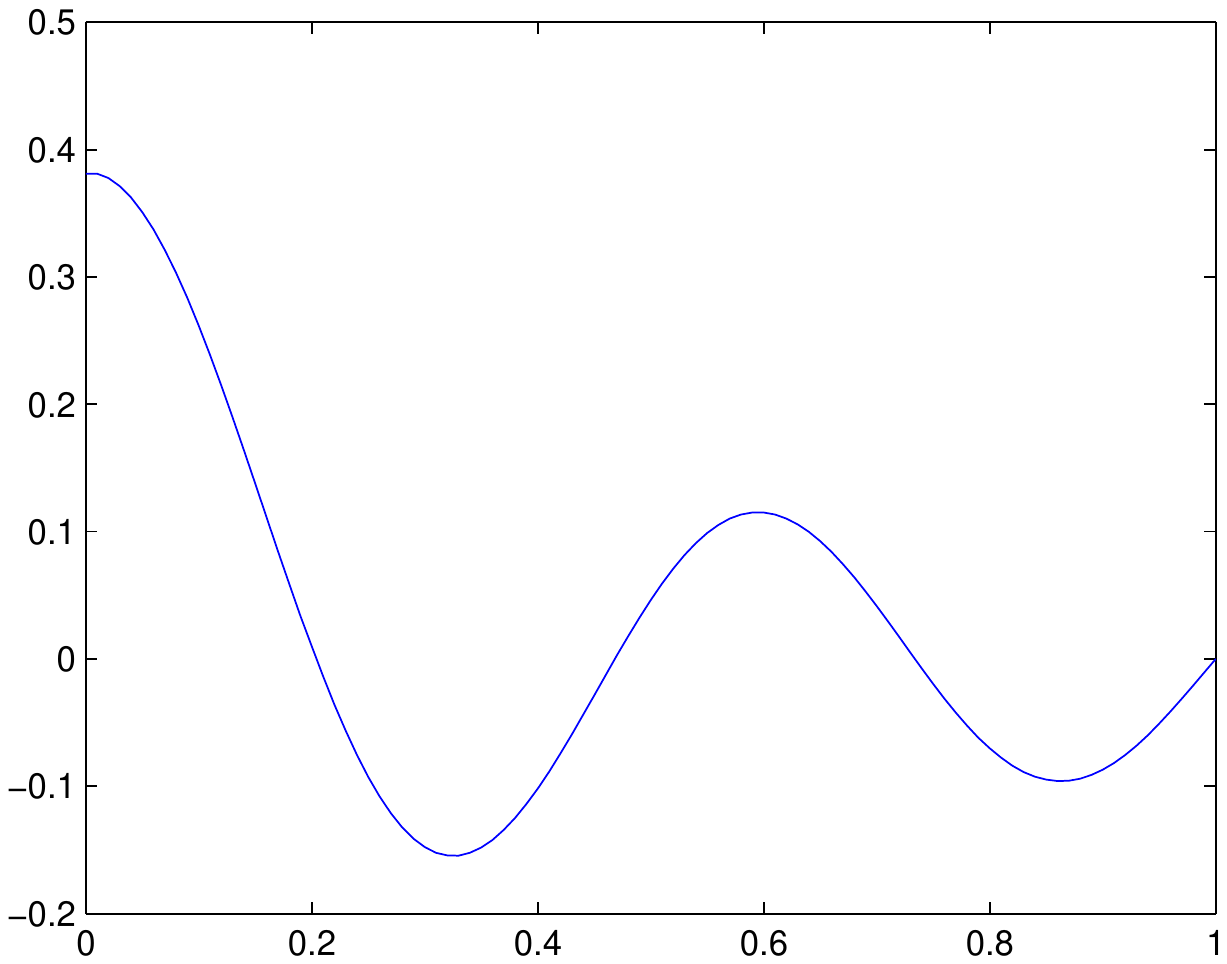}
\vspace{-2cm}
\includegraphics[height=6cm]{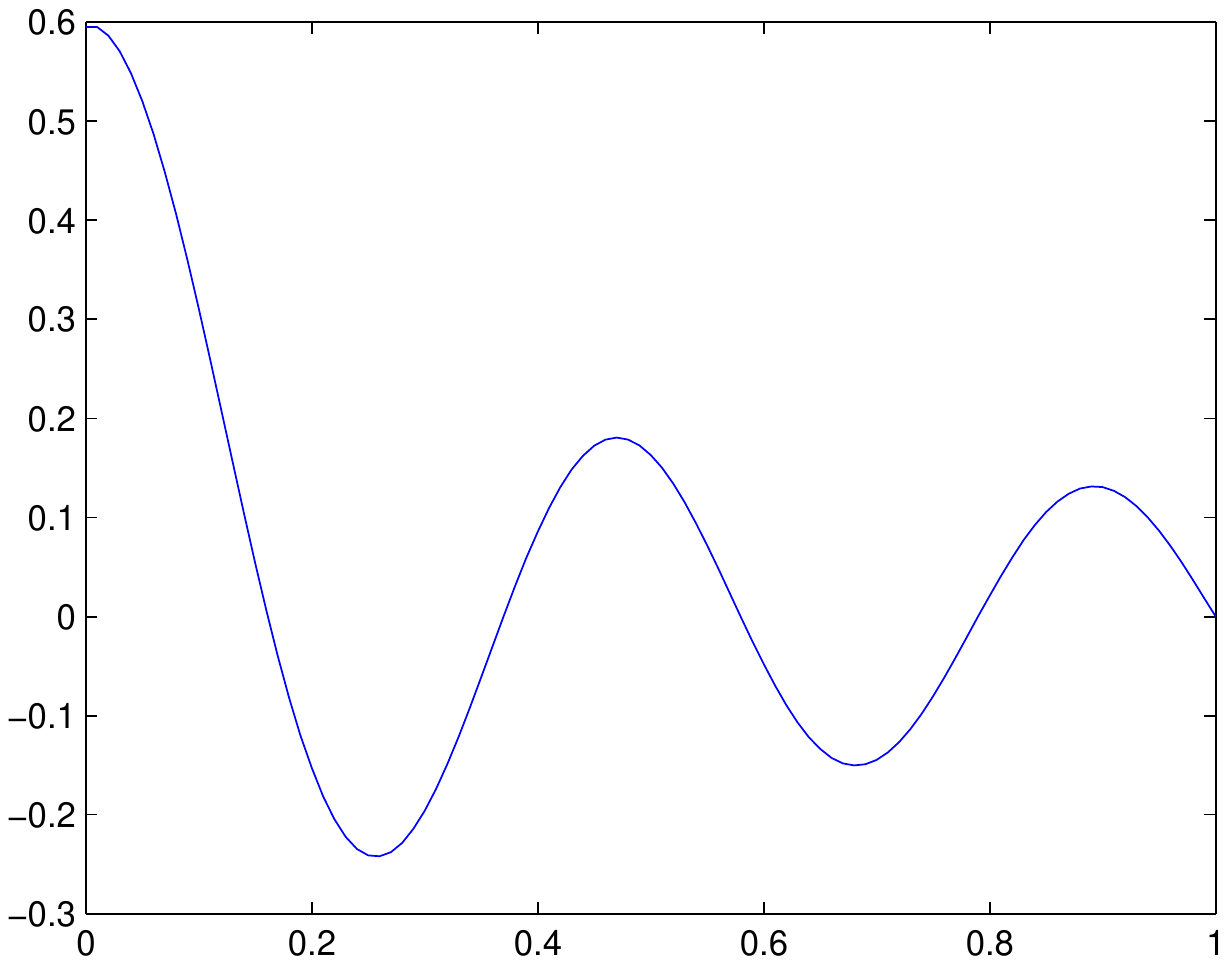}
\includegraphics[height=6cm]{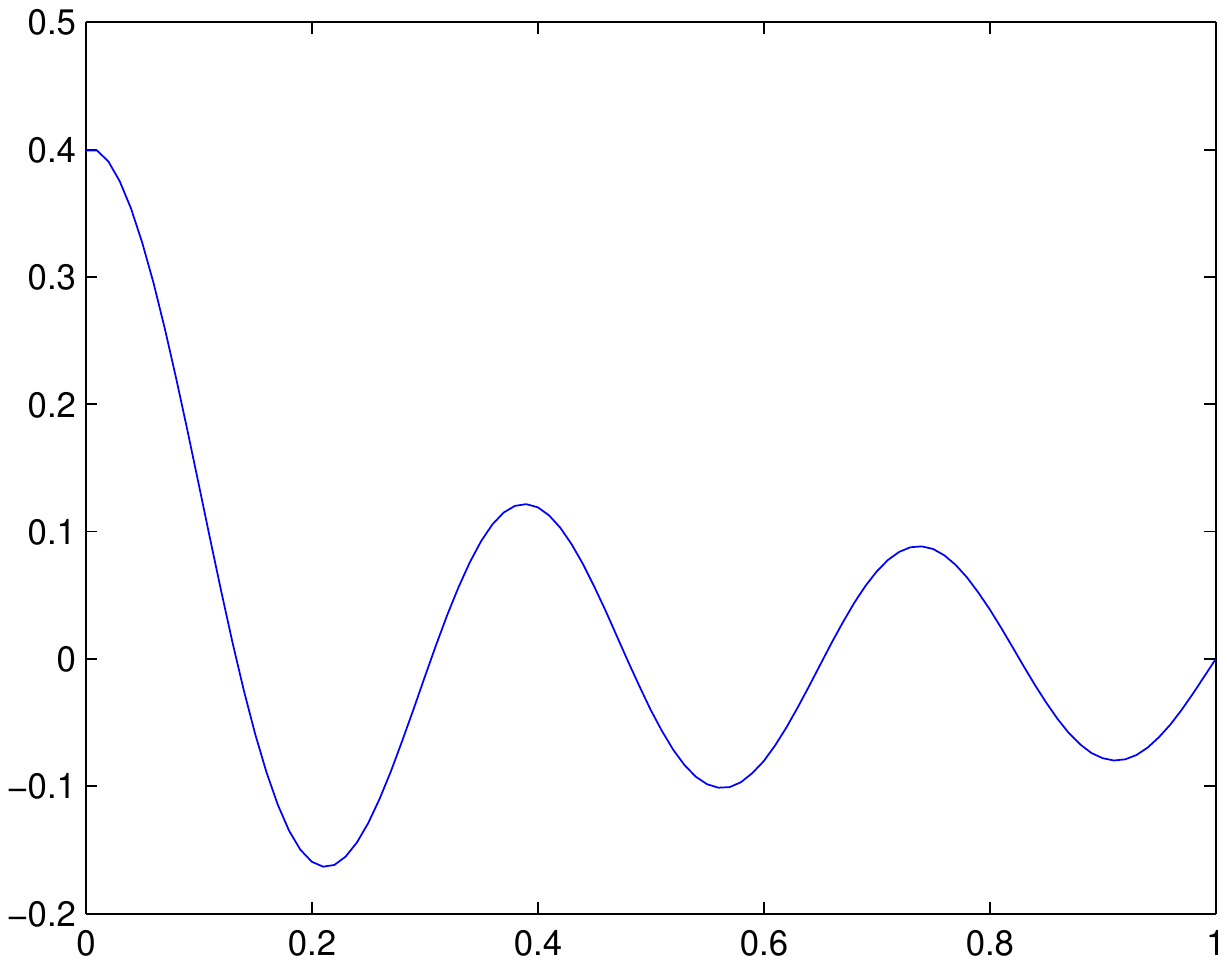}
\includegraphics[height=6cm]{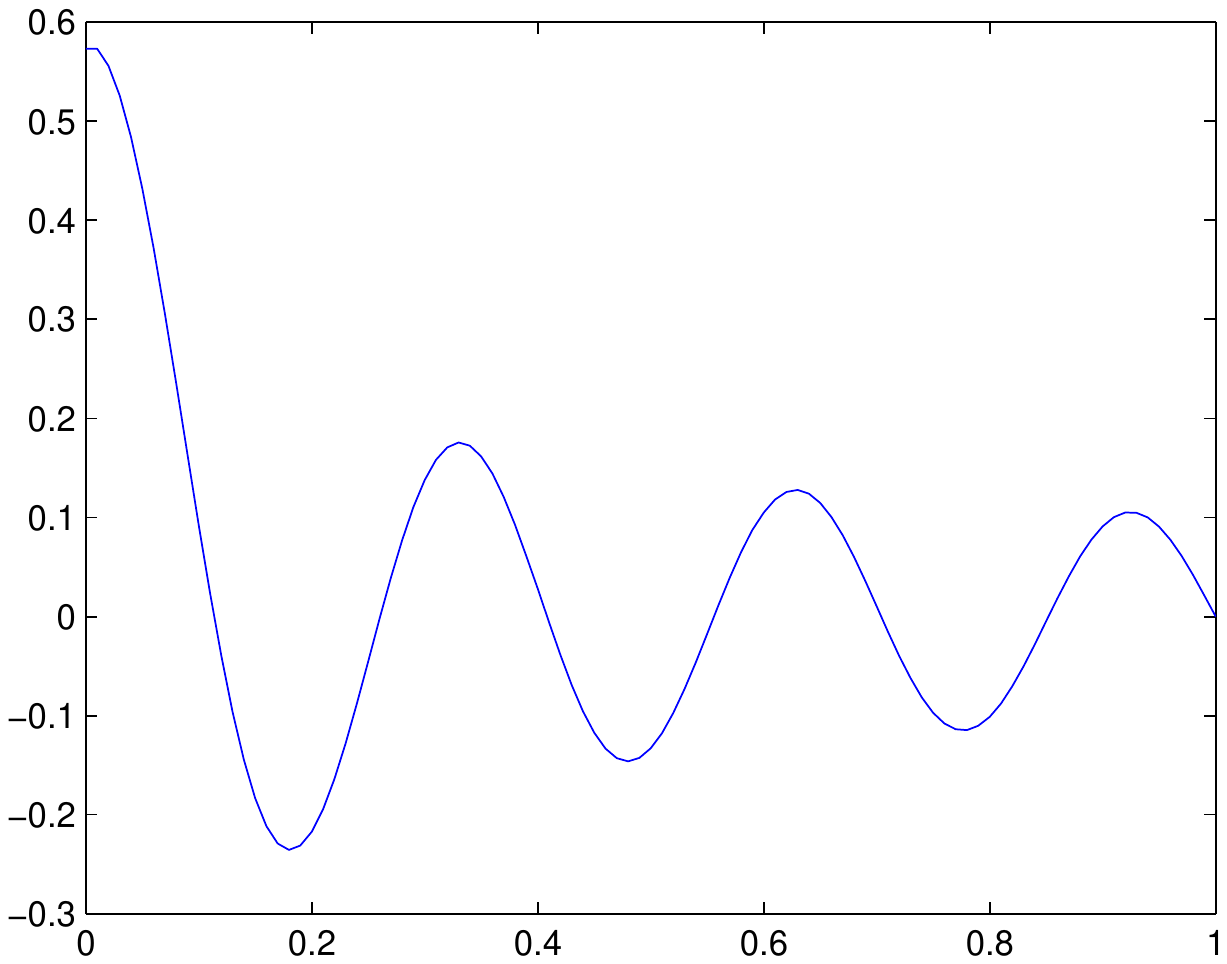}
\includegraphics[height=6cm]{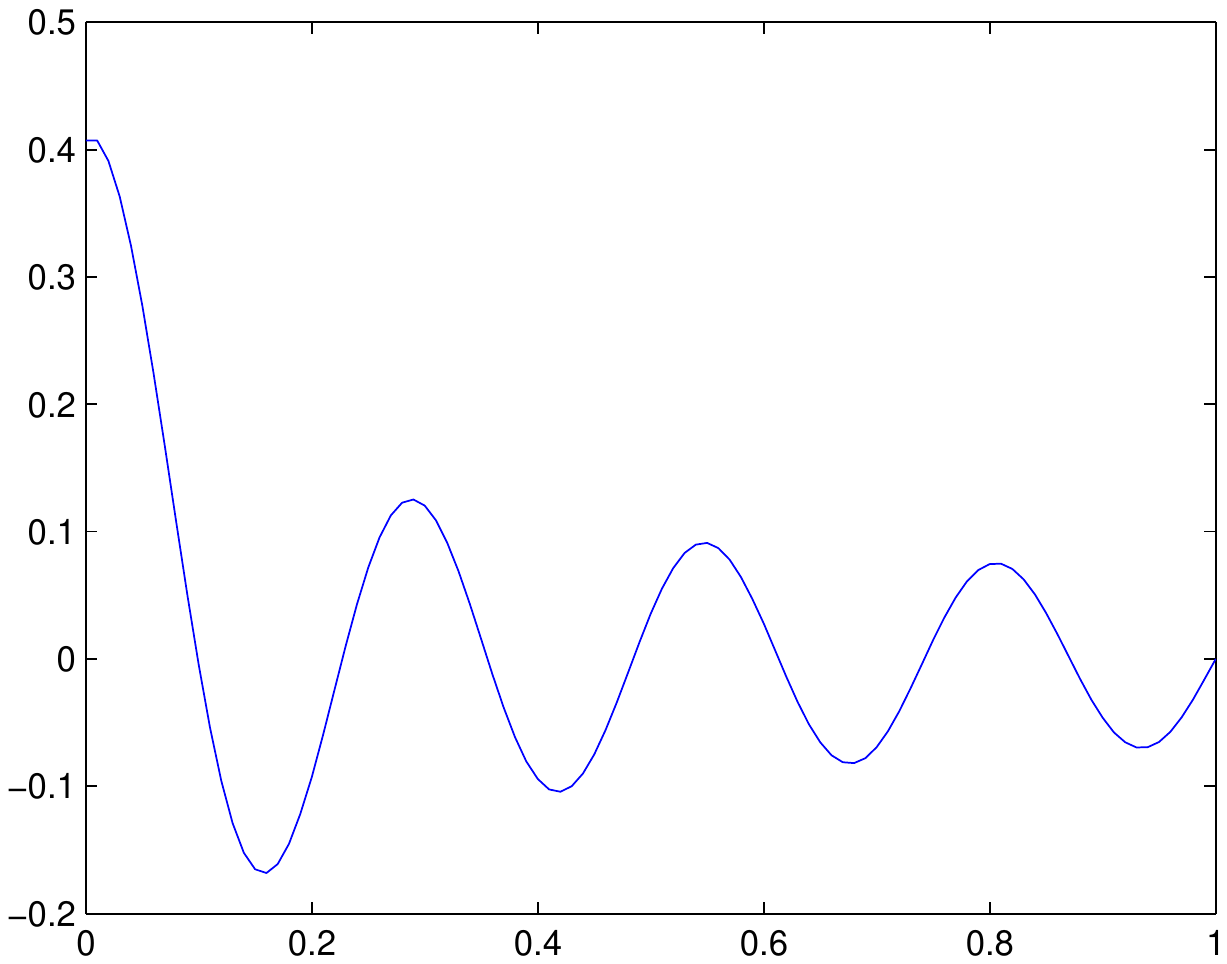}
\end{center}
\caption{First eight Laplacian radial eigenfunctions on the unit disk obtained
from the inverse iteration with shift.}%
\end{figure}

\break

\subsection{Unit square}

The eigenvalues of the unit square $\Omega=[0,1]\times\lbrack0,1]$ are
$\lambda_{n,m}=(n^{2}+m^{2})\pi^{2}$ and the corresponding $L^{\infty}%
$-normalized eigenfunctions are $e_{n,m}=\sin(n\pi x)\sin(m\pi x).$ Hence it
is easy to verify that the spectrum of the function $u\equiv1$ consists
precisely of those eigenvalues $\lambda_{n,m}$ for which both $n$ and $m$ are
odd and that its first three eigenvalues are $\lambda_{1,1}$, $\lambda_{1,3}$
and $\lambda_{3,3}.$ In Table 6 we present exact and approximated eigenvalues
of the Laplacian on this domain. The shift was set $\sigma_{k}:=\lambda
_{k}-0.1$ and a grid containing $201 \times201$ nodes was used. We can see
again that the sequence $\mathcal{R}(v_{n})$ tends to capture only the
eigenvalues $\lambda_{u}^{\sigma}$ that appear on the spectrum of the function
$u\equiv1$, as shown in the last column. Note that the shift $\sigma
_{1,2}:=\lambda_{1,2}-0.1$ is closer to $\lambda_{1,2}$ but, however, the
corresponding sequence $\mathcal{R}(v_{n})$ approaches the correct eigenvalue
$\lambda_{1,1}.$ The same behavior happens with the shifts $\sigma
_{2,2}:=\lambda_{2,2}-0.1$ and $\sigma_{2,3}:=\lambda_{2,3}-0.1$ since they
are closer to $\lambda_{2,2}$ and $\lambda_{2,3},$ respectively, but the
corresponding sequences $\mathcal{R}(v_{n})$ approach to $\lambda_{1,3}.$ The
graphs of the first three eigenfunctions in the spectrum of the unit function
using Algorithm 2 are displayed in Figure 3.

\begin{table}[h]
%\label{squareTable}
%\par
\par
\begin{center}%
\begin{tabular}
[c]{llllll}\hline
$(n,m)$ & $\lambda_{n,m}$ & $\mu_{10}$ & $\gamma_{10}$ & $\mathcal{R}%
(\phi_{10})$ & $\mathcal{R}(v_{10})$\\\hline
$(1,1)$ & $19.7392$ & $19.7388$ & $19.7388$ & $19.7388$ & $19.7388$\\
$(1,2)$ & $49.3480$ & $49.3346$ & $49.3446$ & $49.3446$ & $19.8395$\\
$(2,2)$ & $78.9568$ & $78.9504$ & $78.9504$ & $78.9504$ & $98.6732$\\
$(1,3)$ & $98.6960$ & $98.6796$ & $98.6308$ & $98.6796$ & $98.6796$\\
$(2,3)$ & $128.305$ & $128.285$ & $128.285$ & $128.285$ & $98.7042$\\
$(3,3)\hspace{0.4in}$ & $177.653\hspace{0.4in}$ & $177.620\hspace{0.4in}$ &
$177.620\hspace{0.4in}$ & $177.562\hspace{0.4in}$ & $177.620$\\\hline
\end{tabular}
\end{center}
\caption{Exact and approximated eigenvalues of the Laplacian on the unit
square.}%
\end{table}

In Table 7 we used the sequence $\mu_{n}$ to show the effect of refining the
grid. The shift was chosen $\sigma_{3,3}:=\lambda_{3,3}-0.1$ and 10 iterations
were used. As expected, a finer grid provides a better approximation of the eigenvalue.

\begin{table}[h]
%\label{squareTable}
%\par
\par
\begin{center}%
\begin{tabular}
[c]{lll}\hline
Grid & $\mu_{10}$ & $\epsilon(\mu_{10},\lambda_{3,3})$\\\hline
$100\times100$ & $177.5187$ & $7.6\times10^{-4}$\\
$200\times200$ & $177.6197$ & $1.9\times10^{-4}$\\
$300\times300$ & $177.6382$ & $8.3\times10^{-5}$\\
$400\times400$ & $177.6446$ & $4.7\times10^{-5}$\\
$500\times500$ & $177.6476$ & $3.0\times10^{-5}$\\
$1000\times1000$ & $177.6516$ & $7.4\times10^{-6}$\\
$2000\times2000\hspace{0.4in}$ & $177.6526\hspace{0.4in}$ & $1.9\times10^{-6}%
$\\\hline
\end{tabular}
\end{center}
\caption{Approximated eigenvalues of the Laplacian on the unit square and
relative errors for different grids. The exact eigenvalue is $\lambda_{3,3} =
177.6529$.}%
\end{table}

\begin{figure}[ptb]
\begin{center}
\includegraphics[height=8cm]{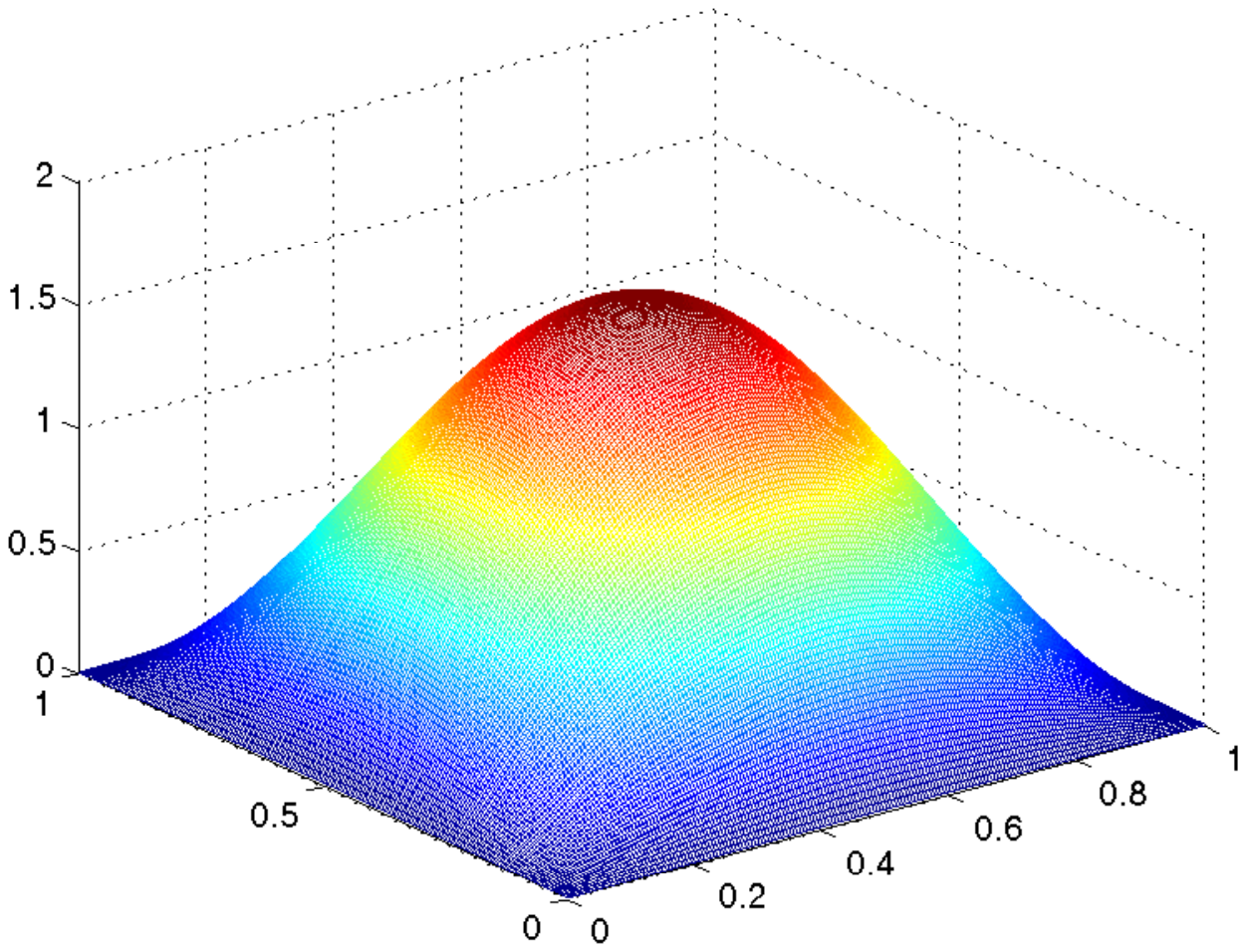} 
\includegraphics[height=8cm]{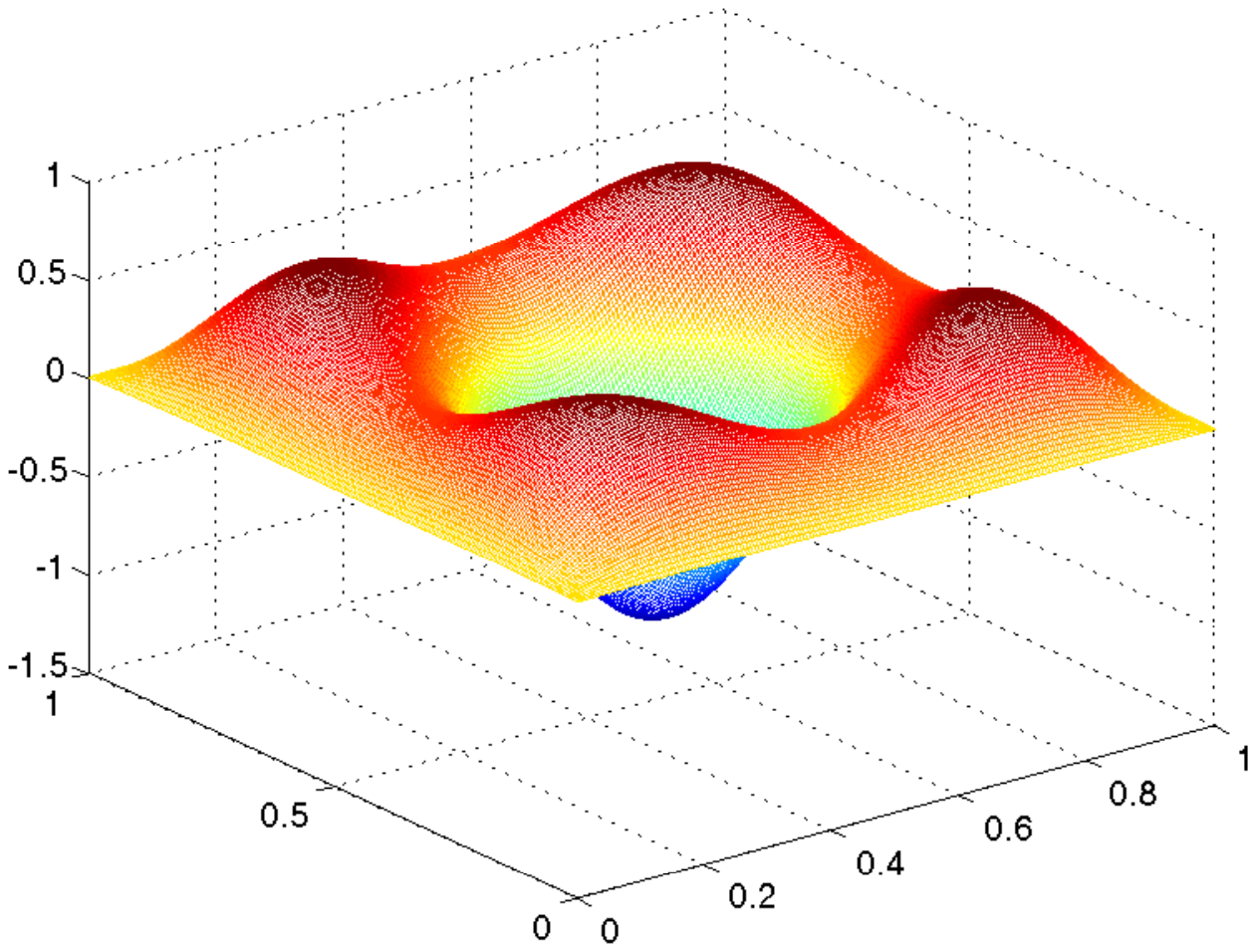}
\includegraphics[height=8cm]{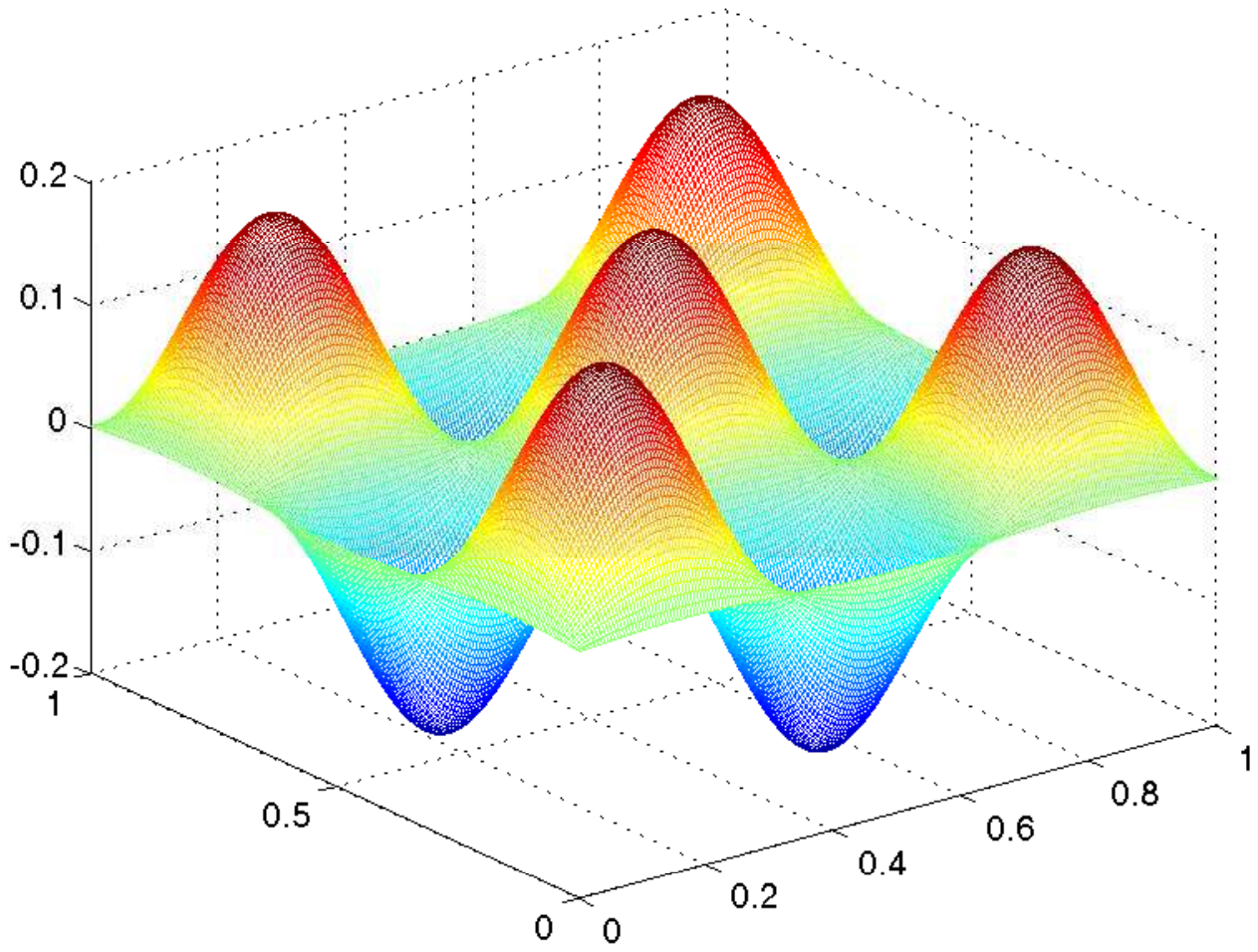}
\vspace{-2cm}
\end{center}
\caption{Graphs of the approximations for the first three projections of the
function $u\equiv1$ on the unit square obtained of the inverse iteration with
shift $\sigma_{n,m}:=\lambda_{n,m}-0.1.$ $e_{u}^{\sigma_{1,1}}$ (left),
$e_{u}^{\sigma_{1,3}}$ (center) and $e_{u}^{\sigma_{3,3}}$ (right).}%
\end{figure}

\section{Final comments\label{final}}

In finite linear algebra, the iterative process itself is often used in order
to generate increasingly better estimates for the eigenvalue at each
iteration, meaning that the approximation obtained at any given iteration is
used as the shift in the next iteration. It turns out that instead of using
the estimates for the eigenvalue obtained in the process, the Rayleigh
quotient of the estimates for the eigenvector obtained at each iteration give
much better approximations for the eigenvalue. Indeed, if the eigenvalues of
the operator or at least very good estimates of them are known in advance,
inverse iteration with shift given by the Rayleigh quotient is the standard
method for computing eigenvalues due to its cubic rate of convergence (see
\cite{Trefenthen-Bau}). It would be only natural to extend such ideas to the
Laplacian, but we were not able to do it. Instead, our (admittedly
preliminary) numerical tests, not shown in this paper, did not indicate
convergence to the correct eigenvalues. As previously discussed, the Rayleigh
quotient may not be a good way to approximate the eigenvalue of high frequency
eigenfunctions unless the grid is much further refined, due to high
oscillations, and the computational cost of using too fine grids can seriously
limit the efficiency of the method. Further investigation is needed. So it
remains an open problem to us if inverse iteration with shift given by the
Rayleigh quotient is a method that can be successfully applied to the Laplacian.

The method described in this paper uses a modification of the Rayleigh
quotient that avoids the computation of gradients. Our numerical tests shown
in Section 6 indicate a greater degree of convergence to the correct
eigenvalues when this form is used. It seems to us that the non-necessity of
calculating the gradients makes our algorithm more numerically stable.

The only reference we could find where the Rayleigh quotient was used in
computing the eigenvalues of the Laplacian, and only for polygonal domains,
was the work \cite{Descloux-Tolley}; however the Rayleigh quotient was only
indirectly used there, as one component of another algorithm and in a very
different way from the direct approach we follow here.

\section{Acknowledgments}

The authors thank the support of FAPEMIG and CNPq - Brazil.

\end{document}